\documentclass[12pt]{amsart}
\usepackage{amsmath, amssymb}

\usepackage{color}

\numberwithin{equation}{section}

\theoremstyle{plain}
\newtheorem*{t-theorem}{Theorem}
\newtheorem{theorem}{Theorem}[section]

\newtheorem{corollary}[theorem]{Corollary}
\newtheorem{lemma}[theorem]{Lemma}

\theoremstyle{definition}
\newtheorem{remark}[theorem]{Remark}
\newtheorem{definition}[theorem]{Definition}
\newtheorem{step}{Step}
\newtheorem*{claim}{Claim}

\def \N {\mathbb{N}}
\def \R {\mathbb{R}}

\def \E {\mathbb{E}}
\def \P {\mathbb{P} \, }

\def \VV {\mathcal{V}}
\def \WW {\mathcal{W}}
\def \LL {\mathcal{L}}
\def \NN {\mathcal{N}}

\def \a {\alpha}
\def \b {\beta}
\def \g {\gamma}
\def \e {\varepsilon}
\def \d {\delta}
\def \D {\Delta}

\def \t {\tau}

\def \Om {\Omega}

\def \etc {,\ldots,}

\def \conv {{\rm conv}}

\def \vol {{\rm vol}}

\def \supp {{\rm supp}}

\def \rowprod {\otimes_r}

\newcommand{\pr}[2]{\langle {#1} , {#2} \rangle}
\newcommand{\norm}[1]{\left \| #1 \right \|}

\begin{document}

\title {Row products of random matrices}
\author {Mark Rudelson}
\address{Department of Mathematics \\
   University of Michigan \\
   Ann Arbor, MI 48109. }
\email{rudelson@umich.edu}
\thanks{Research was supported in part by  NSF grants DMS-0907023 and DMS-1161372.}

\begin{abstract}
  Let $\D_1 \etc \D_K$ be $d \times n$ matrices. We define the row
  product of these matrices as a $d^K \times n$ matrix, whose rows
  are entry-wise products of rows of $\D_1 \etc \D_K$. This
  construction arises in certain computer science
  problems.  We study the question, to which extent
  the spectral and geometric properties of the row product
  of independent random  matrices resemble
   those properties for a $d^K \times n$ matrix with independent
  random entries. In particular, we show that the largest and the
  smallest singular values of these matrices are of the same order,
  as long as $n \ll d^K$.

  We also consider a problem of privately releasing the summary
  information about a database, and use the previous results to
  obtain a bound for the minimal amount of noise, which has to be added to the
  released data to avoid a privacy breach.
\end{abstract}

\maketitle

\section{Introduction}

 This paper discusses spectral and geometric properties of a certain
 class of random matrices with dependent rows, which are constructed
 from random matrices with independent entries. Such constructions
 first appeared in computer science, in the study of privacy
 protection for contingency tables. The behavior of the extreme
 singular values of various random matrices with dependent entries
 has been extensively studied in the recent years \cite{ALPT1}, \cite{ALPT2},
 \cite{KRSU}, \cite{RV1}, \cite{V}. These matrices arise in
 asymptotic geometric analysis \cite{ALPT1}, signal processing
 \cite{ALPT2}, \cite{RV1}, statistics \cite{V} etc.
  The row products studied below have also originated in a computer
  science problem \cite{KRSU}.

 For two matrices with the same number of rows we define the row
 product as a matrix whose rows consist of entry-wise product of the
 rows of original matrices.
 \begin{definition}
  Let  $x$ and $y$ be $1 \times n$ matrices. Denote by $x\rowprod y$ the $1 \times n$ matrix, whose entries
  are products of the corresponding entries of $x$ and $y$: $x \rowprod y(j) =x(j) \cdot y(j)$.
  If $A$ is an $N \times n$ matrix, and $B$ is an $M \times n$ matrix, denote by
  $A \rowprod B$ an $NM \times n$ matrix, whose rows are entry-wise products of the rows of $A$ and $B$:
  \[
    (A \rowprod B)_{(j-1)M+k-1}=A_j \rowprod B_k,
  \]
  where $(A \rowprod B)_l,A_j, B_k$ denote rows of the corresponding matrices.
 \end{definition}
 Row products arise in a number of computer science related problems. They have been introduced in \cite{GC} and studied in \cite{W} in the theory of probabilistic automata. They also appeared in compressed sensing, see \cite{CM} and \cite{GTSV}, as well as in privacy protection problems \cite{KRSU}. These papers use different notation for the row product; we adopt the one from \cite{GTSV}.

  This paper considers spectral and geometric properties of row products of
  a finite number of independent random matrices. The definition above assumes a certain
  order of the rows of the matrix $A \rowprod B$. This order, however, is not important,
  since changing the relative positions of rows of a matrix doesn't affect its eigenvalues
  and singular values. Therefore, to simplify the notation, we will denote the row of
  the matrix $C= A \rowprod B$ corresponding to the rows $A_j$ and $B_k$ by $C_{j,k}$.
  We will use a similar convention for the rows of the row products of more than two matrices.

  Recall that the singular values of $N \times n$
  random matrix $A$ are the eigenvalues of $(A^*A)^{1/2}$ written in
  the non-increasing order: $s_1(A) \ge s_2(A) \ge \ldots \ge s_n(A)
  \ge 0$. The first and the last singular values have a clear
  geometric meaning: $s_1(A)$ is the norm of $A$, considered as a
  linear operator from $\ell_2^n$ to $\ell_2^N$, and if $n \le N$
  and $\text{rank}(A)=n$, then
  $s_n(A)$ is the reciprocal of the norm of $A^{-1}$ considered as a
  linear operator from $\ell_2^N \cap A \R^n$ to $\ell_2^n$. The
  quantity $\kappa(A)=s_1(A)/s_n(A)$, called {\em the condition
  number} of $A$, controls the error level and the rate of
  convergence of many algorithms in numerical linear algebra. The
  matrices with bounded condition number are ``nice'' embedding of
  $\R^n$ into $\R^N$, i.e. they don't significantly distort the
  Euclidian structure. This property holds, in particular, for
  random $N \times n$ matrices with independent centered subgaussian entries having unit variance,
  as long as $N \gg n$.

  Obviously, the row product of several matrices is a submatrix of their tensor product.
  This fact, however, doesn't provide much information about the spectral properties of
  the row product, since they can be different from those of the tensor product.
  In particular, for random matrices, the spectra of $A \otimes B$ and $A \rowprod B$
  are, indeed, very different.
  For example, let $d \le n \le d^2$,
  and consider
  $d \times n$ matrices $A$ and $B$
  with independent $\pm 1$ random values.
   The spectrum of $A \otimes B$ is the
  product of spectra of $A$ and $B$, so   the norm of $A \otimes B$
  will be of the order
  \[
 O \big( (\sqrt{n}+\sqrt{d})^2 \big)= O (n),
  \]
  and the last singular value is
  $
 O \big( (\sqrt{n}- \sqrt{d})^2 \big),
  $
  see \cite{RV2}. From the other side, computer experiments show
  that the extreme singular values of  the row product behave as for
  the $d^2 \times n$ matrix with independent entries, i.e. the first
  singular value is
  \[
    O (d+ \sqrt{n})=O(d),
  \]
   and the last one is  $O (d -
  \sqrt{n})$, see \cite{KRSU}. Based on this data, it was
  conjectured that the extreme singular values of the row product of
  several
  random matrices behave like for the matrices with independent
  entries. This fact was established in \cite{KRSU} up logarithmic
  terms, whose powers depended on the number of multipliers.
  We remove these logarithmic terms in Theorems \ref{t: norm-row
  product} and  \ref{t: smallest singular value} for row
  products of any fixed number of random matrices with independent
  bounded entries.
  To formulate these results more
  precisely, we introduce a class of uniformly bounded random variables,
  whose variances are
 uniformly bounded below. To shorten the notation we summarize their
 properties in the following definition.
 \begin{definition}
  Let $\d>0$. We will call a random variable $\xi$ a $\d$
  random variable if $|\xi| \le 1$ a.s., $\E \xi=0$, and $\E \xi^2 \ge \d^2$.
 \end{definition}
 We start with an estimate of the norm of the row product of random
 matrices with independent $\d$ random entries.

\begin{theorem}  \label{t: norm-row product}
 Let  $\D_1 \etc \D_K$ be  $d \times n$ matrices with
 independent $\d$ random entries.
 Then the $K$-times entry-wise product
 $\D_1 \rowprod \D_2 \rowprod \ldots \rowprod \D_K$ is a  $d^K \times n$ matrix satisfying
  \[
    \P \left(\norm{\D_1 \rowprod \ldots \rowprod \D_K}  \ge C' (d^{K/2}+n^{1/2})
       \right)
    \le \exp \left(-c \left( d+ \frac{n}{d^{K-1}} \right) \right).
  \]
  The constants $C', c$ may depend upon $K$ and $\d$.
\end{theorem}
The paper \cite{KRSU} uses an $\e$-net argument to bound the norm of the row product. This is one of the sources of the logarithmic terms in the bound. To eliminate these terms, we use a different approach. The expectation of the norm is bounded using the moment method, which is one of the standard tools of the random matrix theory. The moment method allows to bound the probability as well. However, the estimate obtained this way would be too weak for our purposes.  Instead, we apply the measure concentration inequality for convex functions, which is derived from Talagrand's measure concentration theorem.

The bound for the norm in Theorem \ref{t: norm-row product} is the same as for a $d^K \times n$ random matrix with bounded or subgaussian
i.i.d. entries, while the probability estimate is significantly weaker than in the independent case.
 Nevertheless, the estimate of Theorem \ref{t: norm-row product} is
optimal both in terms of the norm bound and the probability
(see Remarks \ref{r: optimal norm} and \ref{r: optimal probability} for details). In the important for us case $d^K \ge n$ the
assertion of Theorem \ref{t: norm-row product} reads
  \[
    \P \left(\norm{\D_1 \rowprod \ldots \rowprod \D_K}  \ge C' \sqrt{d^K}
       \right)
    \le \exp \left(-c d \right).
  \]

 It is well-known that with high probability a random $N \times n$ matrix $A$ with
 independent identically distributed bounded centered random entries has a bounded condition number,
  whenever $N \gg n$ (see, e.g.
 \cite{RV}).
 Our next result shows that the same happens for the row products of
 random matrices as well. For the next theorem we need the iterated logarithmic function.

 \begin{definition}
  For $q \in \N$ define the function $\log_{(q)}: (0, \infty) \to \R$
  by induction.
  \begin{enumerate}
    \item $\log_{(1)} t =\max \big(\log t, 1 \big)$;
    \item $\log_{(q+1)} t = \log_{(1)} \big( \log_{(q)} t \big)$.
  \end{enumerate}
 \end{definition}
 Throughout the paper we assume that the constants appearing in
 various inequalities may depend upon the parameters $K,q,\d$, but
 are independent of the size of the matrices, and the nature of
 random variables.

\begin{theorem}  \label{t: smallest singular value}
 Let $K,q,n,d$ be natural numbers. Assume that
 \[
   n \le \frac{c d^K}{\log_{(q)}d}.
 \]
 Let  $\D_1 \etc \D_K$ be  $d \times n$ matrices with
 independent $\d$ random entries.
 Then the $K$-times entry-wise product
 $\D_1 \rowprod \D_2 \rowprod \ldots \rowprod \D_K$ satisfies
 \[
   \P \left(s_n(\D_1 \rowprod \ldots \rowprod \D_K) \le    c' \sqrt{d^K}   \right)
   \le  C \exp \left( -  \bar{c} d  \right).
 \]
\end{theorem}

 This bound, together with the norm estimate above shows that the
 condition number of the row product of matrices with $\d$ random
 entries exceeds a constant with probability $O(\exp(-c d))$.
 While this probability is close to 0, it is much bigger
 than that for a $d^K \times n$ random matrix with independent random entries,
  in which
 case it is of order $\exp (- d^K)$.
  However, it is easy to show
 that this estimate is optimal (see Remarks \ref{r: optimal
 norm} and \ref{r: zero column}). This weak probability bound
 renders standard approaches to singular value estimates unusable.
 In particular, the size of a $(1/2)$ net on the sphere $S^{n-1}$ is
 exponential in $n$, so the union bound in the $\e$-net argument
 breaks down.

 This weaker bound not only makes the proofs more technically
 involved, but also leads to qualitative effects which cannot be
 observed in the context of random matrices with independent
 entries. One of the main applications of random matrices in
 asymptotic geometric analysis is to finding roughly Euclidean or
 almost Euclidean sections of
 convex bodies. In particular, the classical theorem of Kashin
 \cite{Ka} states that a random section of the unit ball of
 $\ell_1^N$ by a linear subspace of dimension proportional to $N$ is
 roughly Euclidean. The original proof of Kashin used a random $\pm
 1$ matrix to construct these sections. The optimal bounds were
 obtained by Gluskin, who used random Gaussian matrices \cite{Gl}.

 The particular structure of the $\ell_1$ norm plays no role
 in this result, and it can be extended to a larger class of convex bodies.
 Let $D \subset \R^N$ be a convex symmetric
 body such that $B_2^N \subset D$ and define the {\em volume ratio}
 \cite{STJ}
 of $D$ by
 \[
  \text{vr}(D)= \left( \frac{\vol (D)}{\vol(B_2^N)} \right)^{1/N}.
 \]
 Assume that the volume ratio of $D$ is bounded: $\text{vr}(D) \le V$.
 Then for a random $N \times n$ matrix $A$ with independent entries satisfying certain conditions,
 \[
  \P (\exists x \in \R^n \norm{Ax}_D \le (cV)^{-\frac{N}{N-n}} N^{1/2} \norm{x}_2)
  \le \exp (-c N).
 \]
 This fact was originally established in \cite{Sz}, and extended in
 \cite{LPRTV} to a broad
 class of random matrices with independent entries.
  However, the volume ratio theorem doesn't
 hold for the row product of random matrices. We show in Lemma
 \ref{l: volume ratio} that there exists a convex symmetric body $D \subset
 \R^{d^K}$ with bounded volume ratio, such that
 \[
 \inf_{x \in S^{n-1}} \norm{\tilde{\D}x}_D \le c (Kd)^{1/2}
 \]
 with probability 1. For $K>1$ this bound is significantly lower
 than $N=d^{K/2}$, which corresponds to the independent entries case.

 Surprisingly, despite the fact that the general volume ratio
 theorem breaks down, it still holds for the original case of the
 $\ell_1$ ball. The main result of this paper is the following
 Theorem.

\begin{theorem}  \label{t: L_1 norm bound}
 Let $K,q,n,d$ be natural numbers. Assume that
 \[
   n \le \frac{c d^K}{\log_{(q)}d}.
 \]
  and let $\D_1 \etc \D_K$ be  $d \times n$ matrices with
 independent $\d$ random entries. Then the $K$-times entry-wise product
 $\tilde{\D}=\D_1 \rowprod \D_2 \rowprod \ldots \rowprod \D_K$ is a  $d^K \times n$ matrix satisfying
 \[
   \P \left( \exists x \in S^{n-1} \ \norm{\tilde{\D}x}_1 \le c' d^K \right)
   \le C' \exp \left( - \bar{c} d  \right).
 \]
\end{theorem}

Note that the results similar to Theorems \ref{t: norm-row product},
\ref{t: smallest singular value}, and \ref{t: L_1 norm bound} remain
valid if the matrices $\D_1 \etc \D_K$ have different numbers of
rows, and the proofs require only minor changes.

 The rest of the paper is organized as follows. In Section \ref{s: non-privacy}
  we consider a privacy protection problem from which the study of
 row products has originated. We derive
 an estimate on the minimal amount of noise needed to avoid a privacy breach
  from Theorem \ref{t: smallest singular value}.
   Section \ref{s: notation}
  introduces necessary notation.
  Section \ref{s:
  outline} contains an outline of the proofs of Theorems \ref{t: norm-row
  product} and \ref{t: L_1 norm bound}.
  Theorem \ref{t: norm-row product}
  is proved in the first part of Section \ref{s: norm}.
   The rest of
  this section and Section \ref{s: Levy}
  develop technical tools needed to prove Theorem \ref{t: L_1 norm bound}.

  In Section \ref{s: chaining} we introduce a new technical method for obtaining lower estimates. The minimal norm of $Ax$ over the unit sphere is frequently bounded via an $\e$-net argument. The implementation of this approach in \cite{KRSU} was one of the main sources of the parasitic logarithmic terms. In Section \ref{s: chaining} the lower bound is handled differently.  The required bound is written as the infimum of a random process. The most powerful method of controlling the {\em supremum} of a random process is to use chaining, i.e. to represent the process as a sum of increments, and control the increments separately \cite{Ta1}. Such method, however, cannot be directly applied to control the {\em infimum} of a positive random process. Indeed, lower estimates for the increments cannot be automatically combined to obtain the lower estimate for the sum. Nevertheless, In Lemma \ref{l: chaining} we develop a variant of a chaining, which allows to control the {\em infimum} of a process.
  This chaining lemma is the major step in proving Theorem \ref{t: L_1 norm bound}, which is presented  in Section \ref{s: lower
  bounds}, where we also derive
  Theorem \ref{t: smallest singular value} from it.

  {\bf Acknowledgement:} the author thanks Dick Windecker and Martin Strauss for pointing out to the papers \cite{GC, W, CM, GTSV}, and the referee for correcting numerous typos in the first version of the paper.

\section{Minimal noise for attribute non-privacy}    \label{s: non-privacy}

 Marginal, or contingency tables are the standard way of releasing
 statistical summaries of data. Consider a database $D$, which we
 view as a $d \times n $ matrix with entries from $\{0,1\}$. The
 columns of the matrix are $n$ individual records, and the
 rows correspond to $d$ attributes of each record. Each attribute is
 binary, so it may be either present, or absent. For any set of $K+1$
 different attributes we release the percentage of records having
 all attributes from this set. The list of these values for all
 $\binom{d}{K+1}$ sets forms the contingency table. In the row product
 notation the contingency table is the subset of coordinates of the
 vector
 \[
   y=\Big( \overset{K +1 \text{ times}}{ \overbrace{D \rowprod \ldots \rowprod D} }  \Big) w,
 \]
 which correspond to all sets of $K+1$ different rows of the matrix
 $D$. Here $w \in \R^n$ is the vector with coordinates $w=(1
 \etc1)$.

 The attribute non-privacy model refers to the situation when $d-1$
 rows of the database $D$ are publicly available, or leaked, and one
 row is sensitive. The analysis of a more general case,
 where there are more than one sensitive
 attribute can be easily reduced to this setting.
 For the comparison of this model with other privacy models see
 \cite{KRSU}, and the references therein.
 Denote the $(d-1) \times n$
 submatrix of $D$ corresponding to non-sensitive attributes by $D'$,
 and the sensitive vector by $x$. Then the coordinates of $y$ contain
 all coordinates of the vector
 \[
  z=\Big( \overset{K \text{ times}}{ \overbrace{D' \rowprod \ldots \rowprod D'} \rowprod \, x^T }  \Big) w
  =\Big( \overset{K \text{ times}}{ \overbrace{D' \rowprod \ldots \rowprod D'} }  \Big)
  x,
 \]
 which correspond to $K$ different rows of the matrix $D'$.
 Hence, if the database $D'$ is generic, then the
 sensitive vector $y$ can be reconstructed from $D'$ and the released vector $z$ by
 solving a linear system. To avoid this privacy breach, the
 contingency table is released with some random noise. This noise
 should be sufficient to make the reconstruction impossible, and at
 the same time, small enough, so that the summary data presented in
 the contingency table would be reliable. Let $z_{\text{noise}}$ be
 the vector of added noise. Let $\bar{D'}$ be the $\binom{D}{K}
 \times n$ submatrix of $D' \rowprod \ldots \rowprod D'$ corresponding
 to all $K$-element subsets of $\{1 \etc n\}$. If the last singular
 value of $\bar{D'}$ is positive, then one can form the left
 inverse $(\bar{D'})^{-1}_L$ of $\bar{D'}$, and
 $\norm{(\bar{D'})^{-1}_L}=s_n^{-1}(\bar{D'})$.
 In this case, knowing the released data $z+z_{\text{noise}}$ we can
 approximate the sensitive vector $x$ by
 $x'=(\bar{D'})^{-1}_L(z+z_{\text{noise}})$. Then
 \[
   \norm{x-x'}_2= \norm{(\bar{D'})^{-1}_L z_{\text{noise}}}_2
   \le \norm{(\bar{D'})^{-1}_L} \cdot \norm{z_{\text{noise}}}_2.
 \]
 Therefore, if $\norm{z_{\text{noise}}}_2 =o(\sqrt{n} \cdot
 s_n^{-1}(\bar{D'}))$,
 then $\norm{x-x'}_2=o(\sqrt{n})$. Since the coordinates of $x$ are
 0 or 1, we can reconstruct $(1-o(1))n$ coordinates of $x$ by
 rounding the coordinates of $x'$. Thus, the lower estimate of
 $s_n^{-1}(\bar{D'})$ provides a lower bound for the norm of the
 noise vector.

 We analyze below the case of a random database. Assume that the
 entries of the database are independent $\{0,1\}$ variables, and
 the entries in the same column are identically distributed. This
 means that the distribution of any given attribute is the same for
 each record, but different attributes can be distributed
 differently. We exclude almost degenerate attributes, i.e. the
 attributes having probabilities very close to 0 or 1.
  In this case bound on the minimal amount of noise
 follows from

\begin{theorem}  \label{t: non-privacy}
 Let $K,q,n,d$ be natural numbers. Assume that
 \[
   n \le \frac{c d^K}{\log_{(q)}d}.
 \]
 Let $0<p'<p''<1$, and let $p_1 \etc p_d$ be any numbers such that
 $p'<p_j<p''$.
  Consider a $d \times n$ matrix  $A$ with
 independent Bernoulli entries $a_{j,k}$ satisfying $\P(a_{j,k}=1)=p_j$
 for all $j=1 \etc d, \ k=1 \etc n$.

 Then the $K$-times entry-wise product
 $\tilde{A}=A \rowprod A \rowprod \ldots \rowprod A$ is a  $d^K \times n$ matrix satisfying
 \[
   \P \left(s_n(\tilde{A}) \le    c' \sqrt{d^K}   \right)
   \le  C' \exp \left( - \bar{c} d  \right).
 \]
 The constants $c,c',C,C'$ may depend upon the parameters
 $K,q,p',p''$.
\end{theorem}

\begin{proof}
 This theorem will follow from Theorem \ref{t: smallest singular
 value}, after we pass to the row product of matrices having
 independent $\d$ random entries. To this end, notice that if an $m
 \times n$ matrix $U'$ is formed from the $M \times n$ matrix $U$ by
 taking a subset of rows, then $s_n(U') \le s_n(U)$.

  Let $d=2Kd'+m$,
 where $0 \le m<2K$. For $j=1 \etc K$ denote by $\D_j^1$ the submatrix of
 $A$ consisting of rows $(2K(j-1)+1) \etc (2K(j-1)+K)$, and by $\D_J^0$
 the submatrix consisting or rows $(2K(j-1)+K+1) \etc 2Kj$.
 Let $D_j^1,D_j^0 \in \R^{d'}$ be vectors with coordinates
  $D_j^1=(p_{2K(j-1)+1} \etc p_{2K(j-1)+K})$ and
 $D_j^0=(p_{2K(j-1)+K+1} \etc p_{2Kj})$.
 Set
 \[
  \D_j= D_j^0 \rowprod \D_j^1- D_j^1 \rowprod \D_j^0.
 \]
 Then $\D_1 \etc \D_K$ are $d' \times n$
 matrices with independent $\d$ random entries for some $\d$
 depending on $p', p''$.

 Let $U_s, \ s=1,2,3$ be $N_s \times n$ matrices, and let $D \in \R^{N_2}$ be a
 vector with coordinates satisfying $|d_j| \le 1$ for
 all $j$. Then for any $x \in \R^n$
 \[
   \norm{(U_1 \rowprod (D^T \rowprod U_2) \rowprod U_3)x}_2
   \le \norm{(U_1 \rowprod  U_2 \rowprod U_3)x}_2.
  \]
  Indeed, any coordinate of $(U_1 \rowprod (D^T \rowprod U_2) \rowprod
  U_3)x$ equals the correspondent coordinate of $(U_1 \rowprod  U_2 \rowprod
  U_3)x$ multiplied by some $d_{j,j}$, so the inequality above
  follows from the bound on $|d_{j,j}|$.
  This argument shows that for any $(\e_1 \etc \e_K) \in \{0,1\}^K$
  \begin{multline*}
    \norm{\Big(\big((D_1^{1-\e_1})^T \rowprod \D_1^{\e_1} \big) \rowprod \ldots
    \rowprod \big( (D_1^{1-\e_K})^T \rowprod\D_K^{\e_K} \big) \Big) x}_2 \\
    \le  \norm{\Big(\D_1^{\e_1} \rowprod \ldots \rowprod \D_K^{\e_K}\Big) x}_2.
  \end{multline*}

 Therefore,
 \begin{align*}
   \norm{\big(\D_1 \rowprod \ldots \rowprod \D_K \big) x}_2
   &\le \sum_{\e=(\e_1 \etc \e_K) \in \{0,1\}^K}
       \norm{\big( \D_1^{\e_j} \rowprod \ldots \rowprod \D_K^{\e_j} \big) x}_2 \\
   &\le 2^K \norm{\big(A \rowprod \ldots \rowprod A\big) x}_2,
 \end{align*}
 because  $\D_1^{\e_j} \rowprod \ldots
 \D_K^{\e_j}$ is a submatrix of $A \rowprod \ldots \rowprod A$.
 Thus, for any $t>0$
 \[
   \P ( s_n(A \rowprod \ldots \rowprod A) <t)
   \le \P (s_n(\D_1 \rowprod \ldots \D_K) < 2^K t).
 \]
 To complete the proof we use Theorem \ref{t: smallest singular value} with $d'$
 in place of $d$, and note that $d \le 3K d'$.
\end{proof}

\section{Notation and preliminary results} \label{s: notation}

 The coordinates of a vector $x \in \R^n$ are denoted by $(x(1) \etc x(n))$.
 Throughout the paper we will intermittently consider $x$ as a vector in $\R^n$ and as an $n \times 1$ matrix.
 The sequence  $e_1 \etc e_n$ stands for the standard basis in $\R^n$.
 For $1 \le p < \infty$ denote by $B_p^n$ the
 unit  ball of the space $\ell_p^n$:
 \[
   B_p^n= \left\{x \in \R^n \mid \norm{x}_p= \left( \sum_{j=1}^n
   |x(j)|^p \right)^{1/p} \le 1 \right\}.
 \]
 By $S^{n-1}$ we denote the Euclidean  unit sphere.

   Denote by $\norm{A}$ the
 operator norm of the matrix $A$, and by $\norm{A}_{HS}$ the
 Hilbert--Schmidt norm:
 \[
   \norm{A}_{HS}= \left(\sum_{j,k} |a_{j,k}|^2 \right)^{1/2}.
 \]
 The volume of a convex set $D \subset \R^n$ will be denoted
 $\vol(D)$, and the cardinality of a finite set $J$ by $|J|$.
 By $\lfloor x \rfloor$ we denote the integer part of $x \in \R$.
 Throughout the paper we denote by $K$ the number of terms in the
 row product, by $q$ the number of iterations of logarithm, and by
 $\d^2$ the minimum of the variances of the entries of random
 matrices.
 $C,c$ etc. denote constants,
 which may depend on the parameters $K,q$, and $\d$,
 and whose
 value may change from line to line.

  Let $V \subset \R^n$ be a compact set, and let $\e>0$. A set
  $\NN \subset V$ is called an $\e$-net if for any $x \in K$ there
  exists $y \in \NN$ such that $\norm{x-y}_2 \le \e$. If $T: \R^n \to
  \R^m$ is a linear operator, and $\NN$ and $\NN'$ are $\e$-nets in
  $B_2^n$ and $B_2^m$ respectively, then
  \[
   \norm{T} \le (1-\e)^{-1} \sup_{x \in \NN} \norm{Tx}_2
   \le (1-\e)^{-2}  \sup_{x \in \NN}  \sup_{y \in \NN'} \pr{Tx}{y}.
  \]
  We will use the following {\em volumetric estimate}. Let $V
  \subset B_2^n$. Then for any $\e<1$ there exists an $\e$-net $\NN
  \subset V$ such that
  \[
    |\NN| \le \left( \frac{3}{\e} \right)^n.
  \]

  We will repeatedly use Talagrand's measure concentration
  inequality for convex functions (see \cite{Ta}, Theorem 6.6, or \cite{L}, Corollary 4.9).
  \begin{t-theorem} [Talagrand]
   Let $X_1 \etc X_n$ be independent random variables with values in
   $[-1,1]$. Let $f: [-1,1]^n \to \R$ be a convex $L$-Lipschitz function,
   i.e.
   \[
    \forall x,y \in [-1,1]^n \ |f(x)-f(y)| \le L \norm{x-y}_2.
   \]
   Denote by $M$ the median of $f(X_1 \etc X_n)$. Then for any
   $t>0$,
   \[
    \P \Big( |f(X_1 \etc X_n)-M| \ge t \Big)
    \le 4 \exp \left( - \frac{t^2}{16 L^2} \right).
   \]
  \end{t-theorem}

 To estimate various norms we will divide the coordinates of a
 vector $x \in \R^n$ into blocks. Let $\pi: \{1 \etc n\} \to \{1 \etc
 n\}$ be a permutation rearranging the absolute values of the
 coordinates of $x$ in the non-increasing order: $|x(\pi(1))| \ge
 \ldots \ge |x(\pi(1))|$. For $l<n$ and $0 \le m$ define
 \[
  N_0=0, \ N_m=\sum_{j=0}^{m-1} 4^j l, \text{ and set }
   I_m = \pi \Big( \{N_m+1 \etc N_{m+1} \} \Big).
 \]
  In other words, $I_0$ contains $l$ largest coordinates of $|z|$,
   $I_1$ contains $4 l$ next largest, etc. We continue as long as
   $I_m \neq \emptyset$. The block $I_m$ will be called the $m$-th
   block  of type $l$ of the coordinates of $x$. Denote $x|_I$ the
   restriction of $x$ to the coordinates from the set $I$.
  We need the following standard
  \begin{lemma}            \label{l: block decomposition}
   Let $b<1$ and let $x \in B_2^n \cap b B_{\infty}^n$. For $l \le
   b^{-2}$ consider blocks $I_0, I_1, \ldots$ of type $l$ of the coordinates of
   $x$. Then
   \[
     \sum_{m \ge 0} |I_m| \cdot
    \norm{x|_{I_m}}_{\infty}^2
    \le 5.
   \]
  \end{lemma}
  \begin{proof}
   Note that the absolute value of any non-zero coordinate
   of $x_{I_{m-1}}$ is greater or equal $\norm{x|_{I_m}}_{\infty}$.
   Hence,
  \begin{align*}
    \sum_{m \ge 0} |I_m| \norm{x|_{I_m}}_{\infty}^2
    &= l \norm{x|_{I_0}}_{\infty}^2
     + 4 \sum_{m \ge 1} |I_{m-1}| \cdot \norm{x|_{I_m}}_{\infty}^2 \\
    &\le l b^2+4  \sum_{m \ge 1}  \norm{x|_{I_{m-1}}}_2^2
    \le 5.
  \end{align*}
  \end{proof}

 The next lemma shows that Theorem \ref{t: L_1 norm bound} cannot be
 extended from $L_1$ norm to a general Banach space whose unit ball
 has a bounded volume ratio.
 \begin{lemma}  \label{l: volume ratio}
 There exists a convex symmetric body $D \subset \R^{d^K}$ such that
 $B_2^{d^K} \subset D$,
 \[
  \left(\frac{\vol(D)}{\vol(B_2^{d^K})} \right)^{1/d^K} \le C
 \]
 satisfying
 \[
  \inf_{x \in S^{n-1}} \norm{(\D_1 \rowprod \ldots \rowprod \D_K)x}_D
  \le c (Kd)^{1/2}
 \]
  for all $d \times n$ matrices $\D_1 \etc \D_K$   with
   entries $1$ or $-1$.
 \end{lemma}
 \begin{proof}
  Set
  \[
    W= \bigcup_{\e_1 \etc \e_K \in \{-1,1\}^d}   \e_1
   \otimes
   \ldots \otimes \e_K
  \]
  and let
  $
   D= \conv \left ((dK)^{-1/2} W,  B_2^{d^K} \right).
  $
  To estimate the volume ratio of $D$ we use Urysohn's inequality
  \cite{P}:
 \[
  \left(\frac{\vol(D)}{\vol(B_2^{d^K})} \right)^{1/d^K}
  \le d^{-K/2} \E \sup_{x \in D} \pr{g}{x},
 \]
 where $g$ is a standard Gaussian vector in $\R^{d^K}$.
 Since
 \[
   D \subset (dK)^{-1/2} \conv (W) +  B_2^{d^K},
 \]
 the right
 hand side of the previous inequality is bounded by
 \[
  1+  d^{-K/2} \cdot (dK)^{-1/2}\E \sup_{x \in W} \pr{g}{x}
  \le 1+ c (dK)^{-1/2} \log^{1/2} |W|,
 \]
 where $|W|$ is the cardinality of $W$. Since $|W|=2^{dK}$, the
 volume ratio of $D$ is bounded by an absolute constant.

 Let $e_1$ be the first basic vector of $\R^n$.
 The lemma
 now follows from the equality
 $(\D_1 \rowprod \ldots \rowprod \D_K)e_1= \e_1 \otimes
   \ldots \otimes \e_K$, where $\e_1 \etc \e_K$ are the first
   columns of the matrices $\D_1 \etc \D_K$.
 \end{proof}

\section{Outline of the proof} \label{s: outline}

 We begin with proving Theorem \ref{t: norm-row product}. We use the
 moment method, which is one of the standard random matrix theory
 tools. To estimate the norm of a rectangular random matrix $A$ with
 centered entries, one considers  the matrix $(A^*A)^p$
 for some large $p \in \N$, and evaluates the expectation of its
 trace using combinatorics. Since $\norm{A}^{2p} \le \text{tr} (A^*A)^p$,
  any
 estimate of the trace translates into an estimate for the norm.
 Following a variant of this approach, developed in \cite{G}, we
 obtain an upper bound for the norm of the row product of
 independent random matrices, which is valid with probability close
 to $1$. However, the moment method alone is insufficient to obtain
 an exponential bound for the probability. To improve the
 probability estimate, we combine the bound for the median of the
 norm, obtained by the moment method, and a measure concentration
 theorem. To this end we extend Talagrand's measure concentration
 theorem for convex functions to the functions, which are {\em
 polyconvex}, i.e. convex with respect to certain subsets of
 coordinates.

 Before tackling the small ball probability estimate for
 \[
  \min_{x \in
   S^{n-1}}  \norm{(\D_1 \rowprod \ldots \rowprod \D_K)x}_1,
  \]
   we consider
 an easier problem of finding a lower bound for \\
 $  \norm{(\D_1 \rowprod \ldots \rowprod \D_{K-1} \rowprod \D_K)x}_1$ for a fixed vector $x
 \in S^{n-1}$. The entries of the row product are not independent, so
 to take advantage of independence, we condition on $\D_1 \etc
 \D_{K-1}$. To use Talagrand's theorem in this context, we have
 to bound the Lipschitz constant of this norm above, and the median
 of it below. Such bounds are not available for all matrices $\D_1 \etc
 \D_{K-1}$, but they can be obtained for ``typical'' matrices,
 namely outside of a set of a small probability.
 Moreover, the bounds will depend on the vector $x$, so to obtain
 them, we have to prove these estimates for all submatrices of the
 row product.
 This is done in
 Sections \ref{s: norm}.2 and \ref{s: norm}.3. Using these results,
 we bound the small ball probability in Section \ref{s: Levy}.
 Actually, we prove a stronger estimate for the {\em Levy
 concentration function}, which is the supremum of the small ball
 probabilities over all balls of a fixed radius.

 The final step of the proof is  combining the individual small ball
 probability estimates to obtain an estimate of the minimal
 $\ell_1$-norm over the sphere. This is usually done by introducing
 an $\e$-net, and approximating a point on the sphere by its
 element. Since the small ball probability depends on the direction
 of the vector $x$, one $\e$-net would not be enough. A modification
 of this method, using several $\e$-nets was developed in
 \cite{LPRT}. However, its implementation for the row products lead
 to appearance of parasitic logarithmic terms, whose degrees rapidly
 grow with $K$ \cite{KRSU}. To avoid these terms, we develop a new {\em chaining
 argument} in Section \ref{s: chaining}. Unlike standard chaining
 argument, which is used to bound the supremum of a random process,
 the method of section \ref{s: chaining} applies to the infimum.

 In section \ref{s: lower bounds} we
 combine the chaining lemma with the Levy concentration function
 bound of Section \ref{s: Levy} to complete the proof of Theorem
 \ref{t: L_1 norm bound}, and derive Theorem \ref{t: smallest singular
 value} from it. We also show that the image of $\R^n$ under the row
 product of random matrices is a Kashin subspace, i.e. the $\ell_1$
 and $\ell_2$ norms are equivalent on this space.

\section{Norm estimates} \label{s: norm}

\subsection{Norm of the matrix}
 We start with a preliminarily estimate of the operator norm of the row product
 of random matrices. To this end we use the moment method, which is
 based on bounding the expectation of the trace of high powers of
 the matrix. This approach, which is standard in the theory of random
 matrices with independent entries, carries over to the row product
 setting as well.
\begin{theorem}  \label{t: norm-moment}
 Let  $\D_1 \etc \D_K$ be  $d \times n$ matrices with
 independent $\d$ random entries.
 Let $p \in \N$ be a number such that $p \le c n^{1/12K}$.
 Then the $K$-times entry-wise product
 $\tilde{\D}=\D_1 \rowprod \D_2 \rowprod \ldots \rowprod \D_K$ is a  $d^K \times n$ matrix satisfying
 \[
   \E \norm{\tilde{\D}}^{2p}
   \le p^{2K+1} n \left( d^{1/2}+n^{1/2K} \right)^{2 pK}.
 \]
\end{theorem}

\begin{proof}

 The proof of this theorem closely follows
 \cite{G}, so we will only sketch it. Denote the entries of the
 matrix $\D_l$ by $\d^{(l)}_{i,j}$, so the entry of the matrix $\tilde{\D}$ corresponding to the product of the entries in the rows $i^{(1)}, i^{(2)} \ldots i^{(K)}$ and column $j$ will be denoted
 $\d_{i_1^{(1)},j}^{(1)} \cdot \ldots \cdot
    \d_{i_1^{(K)},j}^{(K)}$.
      Then
 \begin{align*}
    \E \norm{\tilde{\D}}^{2p}
    &\le \E \text{tr}(\tilde{\D} \tilde{\D}^T)^p \\
    &\le \sum_V \E (\d_{i_1^{(1)},j_1}^{(1)} \cdot \ldots \cdot
    \d_{i_1^{(K)},j_1}^{(K)})
      \cdot (\d_{i_2^{(1)},j_1}^{(1)} \cdot \ldots \cdot
      \d_{i_2^{(K)},j_1}^{(K)})
      \cdot \ldots \\
      & \qquad \ldots \cdot (\d_{i_p^{(1)},j_p}^{(1)} \cdot \ldots \cdot
      \d_{i_k^{(K)},j_p}^{(K)})
      \cdot (\d_{i_1^{(1)},j_p}^{(1)} \cdot \ldots \cdot \d_{i_1^{(K)},j_p}^{(K)}).
 \end{align*}
 Here $V$ is the set of admissible multi-paths, i.e. a sequence of $2p$
 lists $\{(i_{m_1}^{(1)},j_m) \etc (i_{m_K}^{(1)},j_m)\}_{m=1}^{2p}$ such
 that
 \begin{enumerate}
 \item the column number $j_m$ is the same for all entries of the
 list $m$.
   \item the first list is arbitrary;
   \item the entries of the second list are in the same column as
   the  entries of the first list, the entries of the
   third list are in the same rows as the respective entries of the second
   list, etc.;
   \item the entries of the last list are in the same rows as the
   respective entries of the first list;
   \item every entry, appearing in each path, appears at list
   twice.
 \end{enumerate}
 Since the entries of the matrices $\D_1 \etc \D_K$ are uniformly
 bounded, the expectations are uniformly bounded as well, so
 \[
   \E \norm{\tilde{\D}}^{2p} \le  |V|.
 \]
 To estimate the cardinality of $V$ denote by $\b(r_1 \etc r_K,c)$
 the number of admissible multi-paths whose entries are taken from exactly
 $r_1$ rows of the matrix $\D_1$, exactly
 $r_2$ rows of the matrix $\D_2$, etc., and exactly from $c$ columns
 of each matrix. Note that the set of columns through which the path
 goes is common for the matrices $\D_1 \etc \D_K$.
 An admissible multi-path can be viewed as an ordered $K$-tuple of closed
 paths $q_1 \etc q_K$ of length $2p+1$ in the $d \times n$ bi-partite
 graph, such that $q_1(2j)=q_2(2j)= \ldots =q_K(2j)$ for $j=1 \etc
 p$, and each edge is traveled at least twice for each path.
 With this notation we have
 \begin{equation}    \label{norm-moment1}
   \E \norm{\tilde{\D}}^{2p} \le
   \sum_{J} \b(r_1 \etc r_K,c),
 \end{equation}
 where $J$ is the set of sequences
 of natural numbers $(r_1 \etc r_K,c)$ satisfying
 \[
   r_l+c \le p+1 \quad \text{ for each } l=1 \etc K.
 \]
 The inequality here follows from condition (5) above.
 Let $\g(r_1 \etc r_K,c)$ be the number of admissible multi-paths, which
 go through {\em the first} $r_1$ rows of the matrix $\D_1$,
 {\em the first} $r_2$ rows of the matrix $\D_2$, etc., and {\em the
 first} $c$ columns. Then
 \[
   \b(r_1 \etc r_K,c)
   \le \binom{n}{c} \cdot \prod_{l=1}^K \binom{d}{r_l} \cdot
   \g(r_1 \etc r_K,c).
 \]

 We call a closed path of length $2p+1$  path in the
 $d \times n$ bi-partite graph standard if
 \begin{enumerate}
   \item it starts with  the edge $(1,1)$;
   \item  if the path  visits a new
   left (right) vertex, then its number is the minimal among the left
   (right) vertices, which have not yet been visited by this path;
   \item each edge in the path is traveled at least twice.
 \end{enumerate}
 Let $m(r, c)$ is the number of the standard
  paths through
  $r$ left vertices
  and  $c$ right vertices of the bi-partite graph. Then
 \[
  \g(r_1 \etc r_K, c) \le c! \cdot \prod_{l=1}^K  r_l! \cdot
  m(r_l,c).
 \]
  This inequality follows from the fact that all $K$ paths in the
  admissible multi-path visit a new column vertex at the same time,
  so the column vertex enumeration defined by different paths of the
  same multi-path is consistent.
 Combining two previous estimates, we get
 \[
    \b(r_1 \etc r_K,c)
   \le n^c \cdot \prod_{l=1}^K d^{r_l} m(r_l,c).
 \]
 The inequality on page 260 \cite{G} reads
 \[
   m(r,c) \le \binom{p}{r}^2 \cdot p^{12(p-r-c)+14}.
 \]
 Substituting it into the inequality above, we obtain
 \begin{align}   \label{norm-moment2}
   &\sum_{J} \b(r_1 \etc r_K,c) \\
   &\le \sum_{c=1}^p \sum_{r_1+c \le p+1} \ldots \sum_{r_K+c \le p+1}
   n^c \cdot \prod_{l=1}^K d^{r_l}
     \cdot \binom{p}{r_l}^2 \cdot p^{12(p-r_l-c)+14} \notag \\
   &= \sum_{c=1}^p  \prod_{l=1}^K \sum_{r_l=1}^{p+1-c}
     n^{c/K} \cdot d^{r_l}
     \cdot \binom{p}{r_l}^2 \cdot p^{12(p-r_l-c)+14}. \notag
 \end{align}
 To estimate the last quantity note that since $p \le \frac{1}{2} n^{1/12K}$,
 \begin{align*}
   &\sum_{r_l=1}^{p+1-c}
     n^{c/K} \cdot d^{r_l}
     \cdot \binom{p}{r_l}^2 \cdot p^{12(p-r_l-c)+14} \\
   &= p^{2}n^{1/K} \sum_{r_l=1}^{p+1-c} p^{12(p+1-r_l-c)} n^{-(p+1-r_l-c)/K}
       \binom{p}{r_l}^2  d^{r_l} n^{(p-r_l)/K} \\
   &\le  p^{2} n^{1/K} \left( \sum_{r_l=0}^{p}
      \binom{p}{r_l}  (d^{1/2})^{r_l} (n^{1/2K})^{p-r_l} \right)^2 \\
   &= p^{2} n^{1/K} \left( d^{1/2}+n^{1/2K} \right)^{2p}.
 \end{align*}
 Finally, combining this with \eqref{norm-moment1} and
 \eqref{norm-moment2}, we conclude
 \[
   \E \norm{\tilde{\D}}^{2p}
   \le  p^{2K +1}n \left( d^{1/2}+n^{1/2K} \right)^{2pK}
   \le p^{2K +1} n \cdot \left( d^{1/2}+n^{1/2K} \right)^{2pK}.
 \]
\end{proof}

 Applying Chebychev's inequality, we can derive a large deviation
 estimate from the moment estimate of Theorem \ref{t: norm-moment}.
 \begin{corollary}    \label{cor: norm-moment}
  Under the conditions of Theorem \ref{t: norm-moment},
  \[
    \P \left(\norm{\D_1 \rowprod \ldots \rowprod \D_K}  \ge C' (d^{K/2}+n^{1/2})
       \right)
    \le \exp \left(-cn^{\frac{1}{12 K}} \right).
  \]
 \end{corollary}
 \begin{remark} \label{r: optimal norm}
  The bound for the norm appearing in Corollary \ref{cor:
  norm-moment} matches that for a random matrix with centered i.i.d.
  entries. This bound is optimal for the row products as well. To
  see it, assume that the entries of $\D_1 \etc \D_K$ are
  independent $\pm 1$ random variables. Then $\norm{\tilde{\D}e_1}_2= d^{K/2}$.
   Also, if $x \in S^{n-1}$ is
  such that $x(j)=n^{-1/2} \tilde{\d}_{1,j}$, where $\tilde{\d}_{1,j}$
  is an entry in the first row of the matrix $\tilde{\D}$, then
  $\norm{\tilde{\D}x}_2 \ge n^{1/2}$.
 \end{remark}

 More precise versions of the moment method show that the
 moment bound of the type of Theorem \ref{t: norm-row product}
 is valid for bigger values of $p$ as
 well, and lead to more precise large deviation bound. We do not
 pursue this direction here, since these bounds are not powerful
 enough for our purposes.

 Instead, we use the previous corollary to bound the median of the
 norm of $\D_1 \rowprod \ldots \rowprod \D_K$, and apply measure
 concentration. The standard tool for deriving measure concentration
 results for norms of random matrices is Talagrand's measure
 concentration theorem for convex functions.
 However, this theorem
 is not available in our context, since the norm of  $\D_1 \rowprod \ldots \rowprod
 \D_K$ is not a convex function of the entries of  $\D_1 \etc \D_K$.
 We will modify this theorem to apply it to polyconvex functions.
 \begin{lemma}       \label{l: polyconvex}
  Consider a function $F: \R^{KM} \to \R$. For $1 \le k \le K$ and
  $x_1 \etc x_{k-1}, x_{k+1} \etc x_K \in \R^M$ define a function
  $f_{x_1 \etc x_{k-1}, x_{k+1} \etc x_K}: \R^M \to \R$ by
  \[
    f_{x_1 \etc x_{k-1}, x_{k+1} \etc x_K}(x)
    =F(x_1 \etc x_{k-1},x, x_{k+1} \etc x_K)
  \]
  Assume that for all $1 \le k \le K$ and for all
  $x_1 \etc x_{k-1}, x_{k+1} \etc x_K \in B_{\infty}^d$ the functions
  $f_{x_1 \etc x_{k-1}, x_{k+1} \etc x_K}$ are $L$-Lipschitz and
  convex.

  Let $(\e_1 \etc \e_K)=\big( (\nu_{1,1} \etc \nu_{1,M}) \etc
   (\nu_{K,1} \etc \nu_{K,M})\big) \in \R^{KM}$ be a set of
  independent random variables, whose absolute values are uniformly
  bounded by $1$. If
  \[
    \P( F(\e_1 \etc \e_K) \ge \mu) \le 2 \cdot 4^{-K},
  \]
  then for any $t>0$
  \[
    \P( F(\e_1 \etc \e_K) \ge \mu+t)
    \le 4^K \exp \left( - \frac{ct^2}{K^2 L^2} \right).
  \]
 \end{lemma}
 \begin{proof}
  We prove this lemma by induction on $K$. In case $K=1$ the
  assertion of the lemma follows immediately from Talagrand's
  measure concentration theorem for convex functions.

  Assume that the lemma holds for $K-1$.
  Let $F: \R^{KM} \to \R$ be a function satisfying the assumptions
  of the lemma.
   Set
  \[
   \Om=\{(x_1 \etc x_{K-1})\in B_{\infty}^{(K-1)M} \mid \P (F(x_1 \etc x_{K-1},\e_K) >\mu) \ge 1/2
   \}.
  \]
  Then Chebychev's inequality yields
  \begin{equation}\label{Om}
    \P ((\e_1 \etc \e_{K-1}) \in\Om) \le 4^{-(K-1)}.
  \end{equation}
  By
  Talagrand's theorem, for any $(x_1 \etc x_{K-1}) \in B_{\infty}^{(K-1)M} \setminus \Om$
  \[
   \P \left(F(x_1 \etc x_{K-1}, \e_K) \ge \mu+\frac{t}{K} \right)
   \le 2 \exp \left(- \frac{ct^2}{K^2 L^2} \right).
  \]
  Hence,
  \begin{multline*}
   \P \left(F(\e_1 \etc \e_{K}) \ge \mu+\frac{t}{K} \mid (\e_1 \etc \e_{K-1})
   \in B_{\infty}^{(K-1)M} \setminus \Om\right)  \\
   \le 2 \exp \left(- \frac{ct^2}{K^2 L^2} \right).
  \end{multline*}
  Define
  \begin{multline*}
    \Xi= \Big\{x_K \in B_{\infty}^M  \mid
       \P \Big (F(\e_1 \etc \e_{K-1},x_K) \ge
    \mu+\frac{t}{K}\mid \\
      (\e_1 \etc \e_{K-1}) \in B_{\infty}^{(K-1)M} \setminus \Om \Big)  > 4^{-(K-1)} \Big\}.
  \end{multline*}
  The previous estimate and Chebychev's inequality imply
  \[
    \P(\e_K \in \Xi) \le 2 \cdot 4^{K-1} \exp \left(- \frac{ct^2}{K^2 L^2} \right).
  \]
  If $x_K \in \Xi^c$, then combining the conditional probability bound with the estimate \eqref{Om}, we obtain
  \begin{align*}
    &\P(F(\e_1 \etc \e_{K-1},x_K) \ge \mu+\frac{t}{K}) \\
    &\le \P \Big (F(\e_1 \etc \e_{K-1},x_K) \ge
    \mu+\frac{t}{K} \mid  (\e_1 \etc \e_{K-1}) \in B_{\infty}^{(K-1)M} \setminus \Om \Big)
    + \P(\Om)  \\
     &\le 2 \cdot  4^{-(K-1)}.
  \end{align*}
  Hence, applying the induction hypothesis with $\frac{K-1}{K}t$ in place of $t$, we get
  \[
    \P \left(F(\e_1 \etc \e_{K-1},x_K) \ge \mu+\frac{t}{K} +\frac{K-1}{K}t \right)
    \le 4^{K-1}  \exp \left( - \frac{ct^2}{K^2 L^2} \right).
  \]
  Finally,
  \begin{align*}
    &\P( F(\e_1 \etc \e_K) \ge \mu+t) \\
    &\le \P( F(\e_1 \etc \e_K) \ge \mu+t \mid \e_K \in B_{\infty}^M \setminus \Xi) +
    \P(\e_K \in \Xi) \\
    &\le 4^K \exp \left( - \frac{ct^2}{K^2 L^2} \right),
  \end{align*}
  which completes the proof of the induction step.
 \end{proof}

 This concentration inequality
 combined with Corollary \ref{cor: norm-moment}
 allows to establish the correct probability bound for large
 deviations of the norm of the row product of random matrices.

 \begin{proof}[Proof of Theorem \ref{t: norm-row product}]
  For $k=1 \etc K$ let $\e_k \in \R^{dn}$ be the entries of the
  matrix $\D_k$ rewritten as a vector. For any matrices $\D_1 \etc
  \D_{k-1}$, $\D_{k+1} \etc \D_K$ the function
  \[
    f_{\D_1 \etc \D_{k-1}, \D_{k+1} \etc \D_K}(\D_k)=
  \norm{\D_1 \rowprod \ldots \rowprod \D_{k-1} \rowprod \D_k \rowprod \D_{k+1} \rowprod \ldots \rowprod
  \D_K}
  \]
  is convex. Also, since the absolute values of the entries of the
  matrices $\D_1 \etc
  \D_{k-1}, \D_{k+1} \etc \D_K$ do not exceed 1,
  \begin{align*}
    &|f_{\D_1 \etc \D_{k-1}, \D_{k+1} \etc \D_K}(\D_k)
       -f_{\D_1 \etc \D_{k-1}, \D_{k+1} \etc \D_K}(\D_k')| \\
    &\le  \norm{\D_1 \rowprod \ldots \rowprod \D_{k-1} \rowprod (\D_k-\D_k') \rowprod
           \D_{k+1} \rowprod \ldots \rowprod  \D_K} \\
    &\le  \norm{\D_1 \rowprod \ldots \rowprod \D_{k-1} \rowprod (\D_k-\D_k') \rowprod
           \D_{k+1} \rowprod \ldots \rowprod  \D_K}_{HS} \\
    &\le  d^{(K-1)/2} \norm{\D_k-\D_k'}_{HS},
  \end{align*}
  so the Lipschitz constant of this function  doesn't
  exceed $d^{(K-1)/2}$. By Corollary \ref{cor: norm-moment}, we can
  take $\mu= C' (d^{K/2}+n^{1/2})$.
         Applying Lemma \ref{l: polyconvex}
  with $t=C'' (d^{K/2}+n^{1/2})$ finishes the proof.
 \end{proof}
 \begin{remark}  \label{r: optimal probability}
  The probability bound of Theorem \ref{t: norm-row product} is
  optimal.  Indeed, assume first that $d^K \ge n$, and let $\D_1 \etc \D_K$
  be $d \times n$ matrices with independent random $\pm 1$ variables.
  Choose a number $s \in \N$ such that $\sqrt{s}> C$, where $C$ is
  the constant in Theorem \ref{t: norm-row product}, and set
  $x=(e_1+ \ldots +e_s)/\sqrt{s}$.
   With probability
  $2^{-sK \cdot d}$ all entries in the first $s$ columns of these matrices equal
  1, so $\norm{(\D_1 \rowprod \ldots \rowprod \D_K)x}_2 =\sqrt{s} \cdot
  d^{K/2}$.

  In the opposite case, $n> d^K$, set $s=C^2 n/d^K$, where the constant $C$ is the same as above. Then for $x$ defined above we have  $\norm{(\D_1 \rowprod \ldots \rowprod \D_K)x}_2 =\sqrt{s} \cdot
  d^{K/2}= C \sqrt{n}$ with probability at least $2^{-sK \cdot d}=\exp (C' K n / d^{K-1})$.
 \end{remark}

  \subsection{Norms of the submatrices}
 We start with two
 deterministic lemmas. The first
 one is a trivial bound for the norm of the row product of two
 matrices.
 \begin{lemma} \label{l: norm-blocks-trivial}
  Let $U$ be an $M \times n$ matrix, and let $V$ be a $d \times n$ matrix.
  Assume that $|v_{i,j}| \le 1$ for all entries of the matrix $V$.
  Then $\norm{U \rowprod V} \le \sqrt{d} \norm{U}$.
 \end{lemma}

 \begin{proof}
  The matrix $U \rowprod V$ consists of $d$ blocks $U \rowprod v_j, \
  j=1\etc d$, where $v_j$ is a row of $V$. For any $x \in \R^M$
  \[
   \norm{(U \rowprod v_j)x}_2 = \norm{U (v_j \rowprod x^T)^T}_2
   \le \norm{U} \cdot \norm{v_j \rowprod x^T}_2 \le \norm{U} \cdot \norm{
   x}_2.
  \]
  Hence,
  $
   \norm{U \rowprod V}^2 \le \sum_{j=1}^d \norm{U \rowprod v_j}^2 \le d
   \norm{U}^2.
  $
 \end{proof}

 The second lemma is based on the block decomposition of the
 coordinates of a vector.
 \begin{lemma}             \label{l: norm-blocks}
  Let $T: \R^n \to \R^m$ be a linear operator. Set $L=\lceil (1/4)
  \log_2 n \rceil$ and let $1 \le L_0 <L$.
  For $l=1 \etc L$ denote
  \[
   \mathcal{M}_l
   =\{x \in B_2^n \mid |\supp(x)|\le 4^l, \text{\rm and } x(j) \in \{0,
   2^{-l}, -2^{-l} \} \ \text{\rm for all }j \}.
  \]
  Let $b \le 2^{-L_0}$. Then
  \[
    \norm{T: B_2^n \cap b B_{\infty}^n \to B_2^m}
    \le \sqrt{5}
     \left( \sum_{l=L_0}^{L} \max_{z \in \mathcal{M}_l} \norm{T z}_2^2 \right)^{1/2}.
  \]
 \end{lemma}

 \begin{proof}
  Let $x \in B_2^n \cap b B_{\infty}^n$.
   Let $I_0, I_1 \etc I_{L-L_0}$ be blocks of type $4^{L_0}$ of coordinates of
   $x$. Recall that $|I_m|=4^{L_0+m}$.
  If $x_m \neq 0$, set
  \[
    y_m=|I_m|^{-1/2} \cdot \frac{x|_{I_m}}{\norm{x|_{I_m}}_{\infty}},
  \]
   otherwise
  $y_m=0$. Then $\norm{y_m}_{\infty} \le |I_m|^{-1/2}= 2^{-L_0-m}$,
  and $\norm{y_m}_2 \le 1$,
  so $y_m \in \conv (\mathcal{M}_{L_0+m})$ for all $m$. By
  Cauchy--Schwartz inequality,
  \begin{align*}
    \norm{Tx}_2
    &\le \sum_{m=0}^{L-L_0} \norm{T x|_{I_m}}_2
    \le \left( \sum_{m=0}^{L-L_0} |I_m| \cdot \norm{x|_{I_m}}_{\infty}^2
      \right)^{1/2} \cdot
    \left( \sum_{m=0}^{L-L_0} \norm{Ty_m}_2^2 \right)^{1/2} \\
    &\le \left( \sum_{m=0}^{L-L_0} |I_m| \cdot \norm{x|_{I_m}}_{\infty}^2
    \right)^{1/2} \cdot
      \left( \sum_{m=0}^{L-L_0}
         \max_{z \in \mathcal{M}_{L_0+m}} \norm{T z}_2^2 \right)^{1/2}.
  \end{align*}
  The estimate of Lemma \ref{l: block decomposition} completes the proof.
 \end{proof}

 For $k \in \N$ denote by $\WW_k$ the set of all $d^k \times n$
matrices $V$ satisfying
\begin{equation}                                                    \label{d: W_k}
  \norm{V|_J} \le C_k \left(d^{k/2}+\sqrt{|J|} \cdot
  \log^{k/2} \left(\frac{e n}{|J|}\right) \right).
\end{equation}
for all non-empty subsets $J \subset \{1 \etc n\}$. Here $V|_J$
denotes the submatrix of $V$ with columns belonging to $J$, and
$C_k$ is a constant depending on $k$ only. This definition obviously
depends on the choice of the constants $C_k$. These constants will
be defined inductively in the proof of Lemma \ref{l: norm} and then
fixed for the rest of the paper.

 We will prove that the row product of random matrices satisfies
 condition \eqref{d: W_k} with high probability. To this end we need
 an estimate of the norm of a vector consisting of i.i.d. blocks of
 coordinates.

\begin{lemma}  \label{l: point-norm}
  Let $W$ be an $m \times n$ matrix. Let $\theta \in \R^n$ be a vector with
  independent $\d$ random coordinates.
   For $l \in \N$ let $Y_1 \etc Y_l$
  be independent copies of the random variable $Y=\norm{W \theta}$. Then
  for any $s>0$
  \[
    \P \left( \sum_{j=1}^l Y_j^2 \ge 4 l \norm{W}_{HS}^2 +s \right)
    \le 2^l \cdot \exp \left( -\frac{c s}{ \norm{W}^2 } \right).
  \]
\end{lemma}

\begin{proof}
 Note that $F: \R^n \to \R, \ F(x)= \norm{W x}$ is a Lipschitz
 convex function with
 the Lipschitz constant $\norm{W}$. By Talagrand's theorem
 \[
   \P ( |Y - M| \ge t)
   \le 4 \exp \left( - \frac{t^2}{16 \norm{W}^2} \right),
 \]
 where $M= \mathbb{M}(Y)$ is the median of $Y$.
 For $j=1 \etc l$ set $Z_j=|Y_j-M|$. Then the previous
 inequality means that $Z_j$ is a $\psi_2$ random variable, i.e.
 \[
   \E \exp \left( \frac{c' Z_j^2}{\norm{W}^2} \right) \le 2
 \]
 for some constant $c'>0$.
 By the Chebychev inequality and independence of $Z_1 \etc Z_l$,
 \[
   \P \left( \sum_{j=1}^l Z_j^2 >t \right)
   =\P \left( \frac{c'}
         {\norm{W}^2} \sum_{j=1}^l Z_j^2 > \frac{c' t}{\norm{W}^2}
       \right)
   \le 2^l \cdot \exp \left( - \frac{c' t}{\norm{W}^2} \right).
 \]
 Using the elementary inequality $x^2 \le 2 (x-a)^2+2a^2$, valid for
 all $x,a \in \R$, we derive that
 \[
   \P \left( \sum_{j=1}^l Y_j^2 >2 l M^2 +2t \right)
   \le \P \left( \sum_{j=1}^l Z_j^2 >t \right)
   \le 2^l \cdot \exp \left( - \frac{c' t}{\norm{W}^2} \right).
 \]
 By Markov's inequality, $M^2 =\mathbb{M}(Y^2) \le 2 \E Y^2$.
 To finish the proof, notice that since the coordinates of $\theta$
 are independent,
 \[
    \E Y^2
   = \sum_{j=1}^m \sum_{k=1}^n w_{j,k}^2 \cdot \E \theta_k^2 \le \norm{W}_{HS}^2.
 \]
\end{proof}

 The next lemma shows that a ``typical'' row product of random
 matrices satisfies \eqref{d: W_k}.
\begin{lemma}                                               \label{l: norm}
 Let $d,n, k \in \N$ be numbers satisfying $n \ge d^{k+1/2}$.
 Let $\D_1 \etc \D_k$ matrices with independent $\d$ random entries.
 There exist numbers $C_1 \etc C_k>0$ such that
 \[
   \P (\D_1 \rowprod \ldots \rowprod \D_k \notin \WW_k) \le k e^{-cd}.
 \]
\end{lemma}

\begin{proof}
 We use the induction on $k$.
 \begin{step}
  Let $k=1$. In this case $\D_1$ is a matrix with independent $\d$ random  entries.
  For such matrices the result is standard and follows from an easy
  covering argument.
  Let $x \in S^{d-1}$, and let $ y \in S^{n-1} \cap \R^J$. Then
  $\pr{x}{\D_1|_J y}$ is a linear combination of independent $\d$ random variables.
  By Hoeffding's inequality (see e.g. \cite{VW}),
  \[
    \P (|\pr{x}{\D_1|_J y}|>t) \le e^{-ct^2}
  \]
  for any $t \ge 1$.
  Let $J \subset \{1 \etc n\}, \ |J|=m$. Let $\mathcal{N}$ be a
  $(1/2)$-net in $S^{d-1}$, and let $\mathcal{M}$ be a $(1/2)$-net in
  $S^{n-1} \cap \R^J$. Then
  \[
    \norm{\D_1|_J} \le 4 \sup_{x \in \mathcal{N}} \sup_{y \in
    \mathcal{M}} \pr{x}{\D_1|_J y}.
  \]
  The nets $\mathcal{N}$ and $\mathcal{M}$ can be chosen so that
  $|\mathcal{N}| \le 6^d$ and $|\mathcal{M}| \le 6^m$.
  Combining this with the union bound, we get
  \[
    \P( \norm{\D_1|_J} \ge 4t)
    \le |\mathcal{N}| \cdot |\mathcal{M}| \cdot e^{-ct^2}
    \le \exp \left(-ct^2 + (m+d) \log 6 \right)
    \le  e^{-c't^2}
  \]
  provided that $t \ge C (\sqrt{d}+\sqrt{m})$.
  Let
  \[
    t=t_m=\t \cdot (\sqrt{d}+\sqrt{m} \sqrt{\log \frac{e n}{|J|}}),
  \]
   with $\t > C$ to be chosen later, and set $C_1=4\t$. Taking the union
  bound, we get
  \begin{align*}
    \P(\D_1 \notin \WW_1)
    &\le \sum_{m=1}^n \sum_{|J|=m} \P(\norm{\D_1|_J} > 4t_m)
    \le \sum_{m=1}^n \binom{n}{m} e^{-c' t_m^2}  \\
    &\le \sum_{m=1}^n \exp \left[ -c' \t^2 \cdot \left(
              \sqrt{d}+\sqrt{m} \sqrt{\log \frac{e n}{m}} \right)^2
                        +m \log \frac{e n}{m}
                        \right].
  \end{align*}
  We can choose the constant $\t$ so that the last expression doesn't
  exceed $e^{-d}$.
 \end{step}

 \begin{step}
  Let $k>1$, and assume that
  $C_1 \etc C_{k-1}$ are already defined.
   It is enough to find $C_k>0$ such
  that for any $U \in \WW_{k-1}$ with $|u_{i,j}| \le 1$ for all
  $i,j$
  \begin{equation}    \label{U.V}
    \P (U \rowprod \D_k \notin \WW_k) \le e^{-cd}.
  \end{equation}
  Indeed, in this case
  \[
   \P( \D_1 \rowprod \ldots \rowprod \D_k \notin \WW_k
     \mid \D_1 \rowprod \ldots \rowprod \D_{k-1} \in \WW_{k-1})
   \le e^{-cd}.
  \]
  Hence, the induction hypothesis yields
  \begin{align*}
    &\P( \D_1 \rowprod \ldots \rowprod \D_k \notin \WW_k) \\
    &\le \P( \D_1 \rowprod \ldots \rowprod \D_k \notin \WW_k
      \mid \D_1 \rowprod \ldots \rowprod \D_{k-1} \in \WW_{k-1}) \\
   & \quad + \P (\D_1 \rowprod \ldots \rowprod \D_{k-1} \notin \WW_{k-1}) \\
   & \le k e^{-cd}.
  \end{align*}
  Fix $U \in \WW_{k-1}$. To shorten the notation denote $W=U \rowprod \D_k$.
 For $j \in \N$ define $m_j$ as the smallest number $m$ satisfying
 \[
   d^j \le m \log^j \left( \frac{e n}{m} \right).
 \]
 Our strategy of proving \eqref{U.V} will depend on the cardinality
 of the set $J \subset \{1 \etc n\}$ appearing in \eqref{d: W_k}.

  Consider first any set $J$ such that $|J| \le m_{k-1}$.
  By Lemma \ref{l: norm-blocks-trivial},
 \begin{align*}
   \norm{W|_J}
   &\le \sqrt{d} \norm{U|_J} \le \sqrt{d}
       \cdot C_{k-1} (d^{(k-1)/2}+\sqrt{|J|} \log^{(k-1)/2} (e n/|J|)) \\
   &\le 2C_{k-1}d^{k/2},
 \end{align*}
  and so $W$ satisfies the condition $\WW_K$ with $C_k=2 C_{k-1}$
  for all such $J$.

 Now consider all sets $J$ such  that $m_{k-1}< |J| < m_k$.
 The previous argument shows that any vector $y \in S^{n-1}$
 with $|\supp(y)| \le m_{k-1}$ satisfies $\norm{Wy} \le 2C_{k-1}d^{k/2}$.
 Any $x \in S^{n-1}$ can be decomposed as $x=y+z$, where $|\supp(y)| \le m_{k-1}$
  and $\norm{z}_{\infty} \le m_{k-1}^{-1/2}$. Therefore, to prove \eqref{U.V},
   it is enough to show that
 \begin{multline*}
   \P \left (\exists J \subset \{1 \etc n \} \ m_{k-1}<|J| \le m_k
   \ \text{ and} \right. \\
    \left. \norm{W|_J: B_2^n \cap m_{k-1}^{-1/2}B_{\infty}^n \to B_2^{d^k}} > Cd^k \right))
   \le e^{-cd}.
 \end{multline*}
 To this end take any $z \in S^{n-1}$ such that $|\supp(z)| \le m_k$ and
 $\norm{z}_{\infty} \le m_{k-1}^{-1/2}$.
 We will obtain a uniform bound on $\norm{Wz}_2$ over all such $z$,
 and use the $\e$-net argument to derive a bound for $\norm{W|_J}$
 from it.

 Let $M$ be the minimal natural number such that $4^M m_{k-1} \ge m_k$.
 Let $I_0 \etc I_M$ be blocks of type $m_{k-1}$ of the coordinates
 of $z$.
 Since $U \in \WW_{k-1}$, for any $m \le M$
 \[
  \norm{U|_{I_m}}^2 \le C_{k-1} (d^{k-1}+ |I_m| \log^{k-1} (en/|I_m|))
  \le  2C_{k-1}  |I_m| \log^{k-1} (en),
 \]
 because $|I_m| \ge m_{k-1}$.

   Let $\e=(\e_1 \etc \e_n)$ be a row of the
  matrix $\D_k$. Then the coordinates of the vector $W z$
  corresponding to this row form the vector
  $
    (U \rowprod \e) z = (U \rowprod z^T) \e^T.
  $
  Let $U'$ be the $d^{k-1} \times |J|$ matrix defined as
  \[
    U'= (U \rowprod z^T)|_J.
  \]
  The inequality above and Lemma \ref{l: block decomposition} imply
  \begin{align*}
    \norm{U'}^2 \le \sum_{m=0}^M \norm{U|_{I_m}}^2 \cdot \norm{z|_{I_m}}_{\infty}^2
    &\le  2C_{k-1} \log^{k-1} (en) \sum_{m=0}^M |I_m|\cdot \norm{z|_{I_m}}_{\infty}^2 \\
    &\le  10 C_{k-1}
  \log^{(k-1)/2} \left(e n  \right)
  \end{align*}

  Also, since all entries of $U$ have absolute value at most $1$,
  \[
    \norm{U'}_{HS}^2 \le d^{k-1}.
  \]
  The sequence of coordinates of the vector $W z$ consists of
  $d$ independent copies of $U' \e^T_I$.
    Therefore, applying Lemma \ref{l:
  point-norm} with $l=d$ and $s=t d^k$,
  we get
  \begin{align*}
    p(x):
    &=\P (\norm{W z}^2 \ge (4+t) \cdot d^{k})
    \le 2^d \exp \left(- \frac{c t d^k}{\norm{U'}^2} \right) \\
    &\le 2^d \exp \left(- \frac{t d^k}{c'_k \log^{k-1} \left(e n  \right)}
    \right),
  \end{align*}
  where $c'_k= 4C_{k-1}^2/c$.
  By the volumetric estimate, we can construct a $(1/2)$-net $\NN$ for the set
  \[
    E_k:=\{z \in S^{n-1} \mid |\supp(z)| \le m_k, \ \norm{z}_{\infty} \le m_{k-1}^{-1/2} \}
  \]
   in the Euclidean metric, such that
  \[
   |\NN| \le \binom{n}{m_k}6^{m_k} \le \exp \left(2 m_k \log \left( e n \right ) \right).
  \]
  Since $n \ge d^{k+1/2}$, and $m_k \le d^k$, we have $\log (e n) \le 2 k \log(e n/m_k)$, and so
  \[
    m_k \log (e n) \le (2k)^k \frac{d^k}{\log^{k-1} (e n)}.
  \]
  Hence, we can chose the constant $t=t_k$ large enough, so that
  \begin{align*}
   \P ( \exists z \in \NN \mid \norm{Wz}^2 \ge C_k' d^k)
   &\le |\NN| \cdot
       2^d \exp \left(- \frac{t_k d^k}{c'_k \log^{k-1} \left(e n \right)}
       \right) \\
   &\le \exp \left(- \frac{ d^k}{ \log^{k-1} \left(e n  \right)} \right)
  \end{align*}
  with the constant $C_k'=4+t_k$.
      Thus,
  \[
   \P ( \exists z \in E_k  \mid  \norm{Wz}^2 \ge 4 C_k' d^k)
   \le \exp \left(- \frac{ d^k}{ \log^{k-1} \left(e n  \right)} \right),
  \]
  which implies condition \eqref{d: W_k} with $C_k=(4C_{k-1}^2+4C_k')^{1/2}$ for all sets $J$ such that $|J|<m_k$.

  Finally, consider any set $J$ with $|J| \ge m_k$.
   As in the previous case, we can split any vector $x \in S^{n-1}$ as $x=y+z$,
    where $|\supp(y)| \le m_k$ and $\norm{z}_{\infty} \le m_k^{-1/2}$.
  The previous argument shows that with probability greater
  than $1-\exp \big(-d^k/\log^{k-1}( e n) \big)$,
  \[
    \norm{Wy} \le (4C_{k-1}^2+4C_k')^{1/2} d^k
  \]
  for all such $y$. Therefore, it is enough to estimate
  $\max \norm{W|_J z}$ over $z \in B_2^n \cap m_k^{-1/2} B_{\infty}^n$.
  A $(1/2)$-net in the set $B_2^n \cap m_k^{-1/2} B_{\infty}^n$ is
  too big, so following the argument used in the previous case would
  lead to the losses that break down the proof. Instead, we
  will use the sets $\mathcal{M}_l$ defined in Lemma
  \ref{l: norm-blocks} and obtain the bounds for $\max \norm{W|_J z}$
  for each set separately.

  To this end, set $b=1/\sqrt{m_k}$, and let $L_0$ be
  the largest number such that $2^{-L_0} \ge b$.
  Let $l \ge L_0$ and take any $x \in \mathcal{M}_l$.
  Choose any set $I \supset \supp (x)$ such that $|I|=4^l$.
  As in the previous case,
  let $U'$ be the $d^{k-1} \times 4^l$ matrix defined as
  \[
    U'= (U \rowprod x^T)|_I.
  \]
  Since all non-zero
   coordinates of $x$ have absolute value
  $2^{-l}=1/\sqrt{|I|}$, the assumption $U \in \WW_{k-1}$ implies
  \begin{align*}
    \norm{U'} \le \frac{1}{\sqrt{|I|}} \norm{U|_I}
    &\le  \frac{C_{k-1}}{\sqrt{|I|}} \left(d^{(k-1)/2}+\sqrt{|I|} \cdot
  \log^{(k-1)/2} \left(\frac{e n}{|I|} \right) \right) \\
    &\le  2C_{k-1}
  \log^{(k-1)/2} \left(e n \cdot4^{-l} \right)
  \end{align*}
  The last inequality holds since for any $m \ge m_k \ge m_{k-1}$
 \[
  d^{(k-1)/2} \le \sqrt{m} \log^{(k-1)/2}\left( \frac{e n}{m} \right).
 \]
  Also, as before, all entries of $U$ have absolute value at most 1,
  so
  $
    \norm{U'}_{HS}^2 \le d^{k-1}.
  $
  The sequence of coordinates of the vector $W x$ consists of
  $d$ independent copies of $U' \e^T_I$.
   Therefore, applying Lemma \ref{l:
  point-norm}, we get
  \begin{align*}
    \P (\norm{W x}^2 \ge 4d \cdot d^{k-1}+s)
    &\le 2^d \exp \left(- \frac{cs}{\norm{U'}^2} \right) \\
    &\le 2^d \exp \left(- \frac{ s}{c'_k \log^{k-1} \left(e n \cdot4^{-l} \right)} \right)
  \end{align*}
  where $c'_k=4C_{k-1}^2/c$.
  Set
  \[
  s=s(l)
   = 2 c'_k \cdot    4^l \log^k \left(e n \cdot4^{-l} \right).
  \]
  Then $s(l) \ge 2c_k' m_k \log^k (en/m_k) \ge 2 c_k' d^k$, so the
  previous inequality can be rewritten as
  \[
    \P (\norm{W x}^2 \ge  c_k'' s(l))
     \le  \exp \left(- 2 \cdot 4^l \log \left(e n \cdot4^{-l}  \right)
     \right).
  \]
  Hence, the union bound implies that there exists a constant $C_k$
  satisfying
  \begin{align*}
    &\P (\exists l \ge L_0 \ \exists x \in \mathcal{M}_l \
     \norm{W x} > C_k s(l) )
      \\
    &\le \sum_{l=L_0}^{\infty} \binom{n}{4^l} \cdot 3^{4^l}
      \exp \left(- 2 \cdot 4^l \log \left(e n \cdot4^{-l}  \right) \right)
    \le\exp \left(-   4^{L_0} \log \left(e n \cdot4^{-L_0}  \right) \right) \\
    &\le \exp(-d^k).
  \end{align*}
  Define the event $\Om_1$ by
  \[
  \Om_1 = \{\forall l \ge L_0 \ \forall x \in \mathcal{M}_l \ \
     \norm{W x} \le s(l) \}.
  \]
  The previous inequality means that $\P(\Om_1^c) \le  \exp(-d^k)$.

  Assume that the event $\Om_1$ occurs.
  Let $J \subset \{1 \etc n\}$ be such that $|J| \ge m_k$,
  and choose $L'$ so that $4^{L'-1}< |J| \le 4^{L'}$.
  Applying Lemma \ref{l: norm-blocks}
  to $T=W|_J$ and $b= 1/\sqrt{m_k}$, we obtain
  \begin{align*}
    \norm{W|_J: B_2^n \cap m_k^{-1/2} B_{\infty}^n \to B_2^{d^k}}^2
    &\le 5 c_k'' \sum_{l=L_0}^{L'} s(l)
    \le C''_k   4^{L'} \log^{k} \left(\frac{e n}{4^{L'}} \right) \\
    &\le 4 C''_k |J|  \log^{k} \left(\frac{e n}{|J|} \right).
  \end{align*}
 This shows that condition \eqref{d: W_k} holds with
 $C_k=(4C_{k-1}^2+4C_k'+4C_k'')^{1/2}$ for all non-empty sets $J \subset \{1 \etc n\}$. This completes the induction step and the proof of Lemma \ref{l: norm}.
 \end{step}
\end{proof}

\subsection{Lower bounds for the $Q$-norm}
 To obtain bounds for the Levy concentration function below, we need
 a lower estimate for a certain norm of the row product
 of random matrices.
 \begin{definition}
  Let $U=(u_{j,k})$ be an $M \times m$ matrix. Denote
  \[
  \norm{U}_Q=\sum_{j=1}^M \left( \sum_{k=1}^m u_{j,k}^2
  \right)^{1/2}.
  \]
  In other words, $\norm{\cdot}_Q$ is the norm in the Banach space
  $\ell_1^M(\ell_2^m)$.
 \end{definition}
 If $U$ is an $M \times m$ matrix with independent centered entries of unit variance,
 then for any $x \in \R^n$,
 \[
  \E \norm{U \rowprod x^T}_Q \le \sum_{j=1}^M \left(\E \sum_{k=1}^m
    u_{j,k}^2 x^2(k) \right)^{1/2} = M \norm{x}_2.
 \]
   Moreover, if the coordinates of $x$ are
  commensurate, we can expect that a reverse inequality would follow from the Central Limit theorem.
  This observation leads to the following definition.

   Let $\VV_L$ be the set of $d^L \times n$ matrices $A$
 such that for any $x \in \R^n$
 \begin{equation}                   \label{c: large Q norm}
   \norm{A  \rowprod x^T}_{Q} \ge \tilde{c} d^{L} \norm{x}_2.
 \end{equation}
 We will show below that the row product of $L$ independent $d \times n$ random matrices belongs to $\VV_L$ with high probability, provided that the constant $\tilde{c}$ in \eqref{c: large Q norm} is appropriately chosen.
 To this end, consider the behavior of $\norm{(\D_1 \rowprod \ldots \D_L) \rowprod x^T}_{Q}$ for a fixed vector $x \in \R^n$.
 \begin{lemma}          \label{l: Q-norm one vector}
  Let $\D_1 \etc \D_L$ be $d \times m$ random matrices with
  independent $\d$ random entries. Then for any $x \in \R^m$
  \[
    \P \left( \norm{(\D_1 \rowprod \ldots \D_L) \rowprod x^T}_{Q} \le c
    d^L \norm{x}_2 \right)
    \le \exp \left(- \frac{cd^L \norm{x}_2^2}{ \norm{x}_{\infty}^2}
      \right).
  \]
 \end{lemma}

 \begin{proof}
  Without loss of generality, assume that $\norm{x}_{\infty}=1$, so $\norm{x}_2 \ge 1$.
  Let $\a>0$, and let $\nu_1 \etc \nu_m \in [0,1]$ be independent random variables satisfying
  $\E \nu_j \ge \a$ for all $j=1 \etc n$.
    The standard symmetrization and
   Bernstein's inequality \cite{VW} yield
  \[
   \P ( |\sum_{j=1}^m x^2(j) \nu_j - \E \sum_{j=1}^m x^2(j) \nu_j|>t )
   \le 2 \exp \left( - \frac{t^2}{2(\sum_{j=1}^m x^4(j) +t/3)} \right).
  \]
  Setting $t=(\a/2) \norm{x}_2^2$, and using $\norm{x}_{\infty} \le 1$, we get
  \[
   \P \left(\sum_{j=1}^m x^2(j) \nu_j < \frac{\a}{2} \norm{x}_2^2 \right)
   \le 2 \exp \left( - \frac{\a^2}{16} \norm{x}_2^2 \right).
  \]
  Applying the previous inequality to the random variable $Y_i, \ i=1 \etc d^L$,
  which is the $\ell_2$-norm of a row of the matrix $(\D_1 \rowprod \ldots \D_L) \rowprod x^T$,
  we obtain $\P(Y_i < c \norm{x}_2) \le 2 \exp(-c' \norm{x}_2^2)$.
  Let $0<\theta<1$.
   If
   \[
      \norm{(\D_1 \rowprod \ldots \D_L) \rowprod x^T}_{Q}
      =\sum_{i=1}^{d^L} Y_i \le \theta \cdot    d^L \norm{x}_2,
   \]
   then $Y_i < c \norm{x}_2$ for at least $(1-\theta) d^L$ numbers $i$. Hence,
  \begin{align*}
   &\P \left( \norm{(\D_1 \rowprod \ldots \D_L) \rowprod x}_{Q} \le \theta c
    d^L \norm{x}_2 \right) \\
   & \le \binom{d^L}{\lfloor (1-\theta) d^L \rfloor} \exp (- c(1-\theta) d^L \norm{x}_2^2
      ) \\
   & \le \exp \left( - d^L \Big( c (1- \theta) \norm{x}_2^2- \theta \log \frac{e}{\theta} \Big) \right)
         \le \exp (- (c/2) d^L \norm{x}_2^2    ),
  \end{align*}
  if $\theta$ is small enough.
 \end{proof}

 We will use Lemma \ref{l: Q-norm one vector} to show that the row
 product $\D_1 \rowprod \ldots \rowprod \D_{K-1}$ satisfies condition
 \eqref{c: large Q norm} with high probability.

 \begin{lemma}  \label{l: Q norm}
 There exists a constant $\tilde{c}>0$ for which the following holds. Let $K>1$, and let $n \le d^K$.
  For $d \times n$ matrices $\D_1 \etc \D_{K-1}$ be matrices with independent $\d$ random
  entries
  \[
    \P(\D_1 \rowprod \ldots \rowprod \D_{K-1} \notin \VV_{K-1})
    \le \exp(-cd^{K-1}).
  \]
 \end{lemma}

 \begin{proof}
   Denote for shortness $\bar{\D}= \D_1 \rowprod
  \ldots \rowprod \D_{K-1}$.
  To conclude that $\bar{\D} \in \VV_{K-1}$, it is enough to show that condition \eqref{c: large Q norm} holds
  for any $x \in S^{n-1}$.

  For $x \in S^{n-1}$ denote by $\Om(x)$ the set of matrices $A$ such that
  $
     \norm{A  \rowprod x}_{Q}
    \le c
    d^{K-1}.
  $
  For $L=K-1$ Lemma  \ref{l: Q-norm one vector} yields
  \begin{equation}    \label{Q norm-1}
   \P(\bar{\D} \in \Om^c(x) )
    \le \exp \left(- \frac{cd^{K-1} }{2 \norm{x}_{\infty}^2}
      \right).
  \end{equation}
  As the first step in proving the lemma,
  we will show that for $A=\bar{\D}$ condition \eqref{c: large Q
  norm} holds for all $x$ from some subset of the sphere. More
  precisely, we will prove the following claim.
  \begin{claim}
   Let $a>0$ and $m \le n$.
   Denote
   \[
     S(a,m)= \{x \in S^{n-1} \mid \norm{x}_{\infty} \le a,  \
     |\supp(x)| \le m \}.
   \]
   If $a^2 m \log d < C d^{K-1}$, then
   \[
     \P( \bar{\D} \notin \bigcap_{x \in S(a,m)}\Om(x))
     \le \exp \left(- \frac{c' d^{K-1}}{a^2} \right).
   \]
  \end{claim}
  It is enough to prove the claim for $0<a \le 1$.
  Note that if $0 \le |y(j)| \le |x(j)|$ for any $j=1 \etc k$, then
  $
      \norm{\bar{\D} \rowprod y^T}_{Q}
    \le \norm{\bar{\D} \rowprod x^T}_{Q}.
  $
  Hence, to prove the claim,
  it is enough to construct a set $\NN$ of vectors $y \in B_2^n \setminus (1/2)B_2^n$
  such that for any $x \in S(a,m)$ there is $y \in \NN$ with $|y(j)|
  \le |x(j)|$ for all $j$ and
   \[
     \P( \bar{\D} \notin \bigcap_{y \in \NN}\Om(y))
     \le \exp \left(- \frac{c' d^{K-1}}{a^2} \right).
   \]
   Set
   \[
    \NN= \left \{ y \in \left( \frac{1}{2 \sqrt{m}} \right) \mathbb{Z}^n
    \mid |\supp(y)| \le m,
         \ \norm{y}_{\infty} \le a \text{ and } \frac{1}{2} \le \norm{y}_2 \le 1 \right\}.
   \]
  By the volumetric considerations
   \[
     |\NN| \le \binom{n}{m} C^m \le \exp(c m \log n)
     \le \exp( C' m \log d),
   \]
   since $n \le d^K$. For  $x \in S(a,m)$ consider the vector $y$ with
   coordinates $y(j) = (1/2\sqrt{m}) \cdot \lfloor 2\sqrt{m} |x(j)|
   \rfloor$. Then $|y(j)| \le |x(j)|$, and $\norm{y}_2 \ge 1- \norm{x-y}_2 \ge 1/2$, so  $y \in \NN$.
   By the union bound and \eqref{Q norm-1},
   \[
     \P( \bar{\D} \notin \bigcap_{y \in \NN}\Om(y))
     \le |\NN| \exp \left(- \frac{cd^{K-1}}{2 a^2}
     \right).
   \]
   The claim now follows from the assumption $a^2 m \log d \le C d^{K-1}$
   for a suitable constant $C$.

   The lemma can be easily derived from the claim. For $a$ and $m$
   as above denote $\Om(a,m)=\bigcap_{x \in S(a,m)}\Om(x)$.
   Set
   \[
    a_i=3 d^{(1-i)K/6}, \quad m_i=\min \left( d^{iK/3}, n \right), \quad i=1,2,3.
   \]
   Then $m_3=n$, and the condition $a_i^2 m_i \log d \le Cd^{K-1}, \ i=1,2,3$ is satisfied.
    Set
   \[
     \VV=  \bigcap_{i=1}^3 \Om(a_i,m_i).
   \]
   By the claim, $\P(\VV^c) \le \exp (-cd^{K-1})$.

   Assume now that $\bar{\D} \in \VV$.
   Using the non-increasing rearrangement of $|x(j)|$,
   we can decompose any $x \in S^{n-1}$ as $x=
   x_1+x_2+x_3$,where $x_1, x_2 ,x_3$ have disjoint supports,
    $|\supp(x_i)| \le m_i$, $\norm{x_i}_{\infty} \le
   a_i/3$.
    By the triangle
   inequality,  $\norm{x_i}_2 \ge 1/3$ for some $i$.
   Thus,
   \begin{align*}
      \norm{\bar{\D} \rowprod x^T}_Q
     \ge \norm{\bar{\D} \rowprod x_i^T}_Q
     \ge \norm{\bar{\D} \rowprod \frac{x_i^T}{\norm{x_i}_2}}_Q \cdot
     \frac{1}{3}
     \ge \frac{c}{3} d^{K-1},
   \end{align*}
   since $x_i/\norm{x_i}_2 \in S(a_i, m_i)$. This proves the Lemma with $\tilde{c}=c/3$.
 \end{proof}

\section{Bounds for the Levy concentration function}
    \label{s: Levy}

\begin{definition}
  Let $\rho>0$.
  Define the Levy concentration function of a random vector $X \in \R^n$ by
  \[
    \LL_1 (X,\rho) =\sup_{x \in \R^n} \P (\norm{X-x}_1 \le \rho).
  \]
\end{definition}
Unlike the standard definition of the Levy concentration function,
we use the $\ell_1$-norm instead of the $\ell_2$-norm. We need the
following standard
\begin{lemma}                                  \label{l: levy conc}
  Let $X \in \R^n$ be a random vector, and let $X'$ be an
  independent copy of $X$. Then for any $\rho>0$
  \[
    \LL_1(X,\rho) \le \P^{1/2} (\norm{X-X'}_1 \le 2 \rho).
  \]
\end{lemma}

 \begin{proof}
  Let $y \in \R^n$ be any vector. Then
  \begin{align*}
    \P^2 (\norm{X-y}_1 \le \rho)
    &= \P (\norm{X-y}_1 \le \rho \text{ and } \norm{X'-y}_1 \le \rho)
    \\
    &\le \P(\norm{X-X'}_1 \le 2 \rho).
  \end{align*}
  Taking the supremum over $y \in \R^n$ proves the Lemma.
 \end{proof}
In the next lemma, we bound the Levy concentration function using
Talagrand's inequality, in the same way it was done in the proof of
Lemma \ref{l: point-norm}.

\begin{lemma}                            \label{l: Levy one vector}
 Let $U=(u_{i,j})$
  be any $N \times n$  matrix, and let $\e=(\e_1 \etc \e_n)^T$ be
  a vector with independent $\d$ random coordinates.
  Then for any $x \in \R^n$
  \begin{equation}   \label{i: Levy one vector}
    \LL_1 \left( (U \rowprod \e^T) x , c \norm{U \rowprod x^T}_{Q}
    \right)
   \le 2 \exp \left( - c'\frac{ \norm{U \rowprod x^T}_{Q}^2}{N \norm{U \rowprod x^T}^2} \right).
  \end{equation}
\end{lemma}

\begin{proof}
  Note that $(U \rowprod \e^T) x=(U \rowprod x^T) \e$.
  Let $\e_1' \etc \e_n'$ be
  independent copies of $\e_1 \etc \e_n$.
  Applying Lemma
  \ref{l: levy conc}, we obtain for any $\rho>0$
  \begin{equation}                               \label{Levy one vector-1}
    \LL_1 \left((U \rowprod x^T) \e, \rho
    \right)
    \\
   \le \P^{1/2} \left( \norm{(U \rowprod x^T) (\e-\e')}_1 \le
       2 \rho
    \right).
  \end{equation}
  Consider a function $F:\R^{n} \to \R$, defined
  by
  \[
    F(y)=\norm{(U \rowprod x^T) y}_1,
  \]
  where $y \in \R^n$. Then $F$ is a convex
  function with the
  Lipschitz constant
  $L\le  \norm{U \rowprod x^T: B_2^n \to B_1^N}\le \sqrt{N} \norm{U \rowprod x^T}$.

  By Talagrand's measure concentration theorem
  \[
    \P(|F(\e-\e')-\mathbb{M}(F)|>s)
     \le 4 \exp \left( - \frac{cs^2}{L^2} \right),
  \]
  where $\mathbb{M}(F)$ is a median of $F$,
  considered as a function on $\R^n$ equipped with
  the probability measure defined by the vector $\e-\e'$.
  This tail estimate
  implies
  \[
    |\mathbb{M}(F)- \E F|
    \le c_1 L \le c_1 \sqrt{N} \norm{U \rowprod x^T}.
  \]
  By Lemma 2.6 \cite{RV} we have
  \begin{align*}
    \E F
    &= \E \sum_{i=1}^N \left| \sum_{j=1}^n u_{i,j} x(j) \cdot
    (\e(j)-\e'(j)) \right|
    \ge c_2  \sum_{i=1}^N \left( \sum_{j=1}^n u_{i,j}^2 x^2(j)
    \right)^{1/2}  \\
    &= c_2 \norm{(U \rowprod x^T)}_Q.
  \end{align*}
  Note that if the constant $c'$ in the formulation of the lemma is
  chosen small enough, we may assume that
  $2 c_1 \sqrt{N} \norm{U \rowprod x^T} \le c_2 \norm{(U \rowprod x^T)}_Q$.
  Indeed, if this inequality does not hold, the right-hand side of
  \eqref{i: Levy one vector} would be greater than 1.
   Combining the
  previous estimates yields $\mathbb{M}(F) \ge  (c_2/2) \norm{(U \rowprod x^T)}_Q$.
   Hence,
  \begin{align*}
   &\P \left(  \norm{(U \rowprod x^T) (\e-\e')}_1
       \le \frac{c_2}{4} \norm{U \rowprod x^T}_Q \right) \\
   &\le \P \left(|F(\e-\e')- \mathbb{M}(F)|
       \ge \frac{1}{4} \mathbb{M}(F)\right)  \\
   &\le 4 \exp \left( - \frac{c \mathbb{M}^2(F)}{L^2} \right)
   \le 4 \exp \left( - c'\frac{ \norm{U \rowprod x^T}_{Q}^2}{N \norm{U \rowprod x^T}^2} \right).
  \end{align*}
  This inequality and \eqref{Levy one vector-1}, applied with $\rho=\frac{c_2}{8} \norm{(U \rowprod x^T)}_Q$, finish the proof.
\end{proof}

 For the next result we need the following standard Lemma.
  \begin{lemma}     \label{l: trivial-tensorization}
     Let $s_1 \etc s_d$ be independent non-negative random
     variables such that $\P(s_j \le R) \le p$ for all $j$. Then
     \[
       \P \left(\sum_{j=1}^d s_j \le \frac{1}{2} R d \right)
       \le (4p)^{d/2}.
     \]
  \end{lemma}
  \begin{proof}
    If $\sum_{j=1}^d s_j \le \frac{1}{2}R d$, then $s_j \le R$ for at least
    $d/2$ numbers $j$.
  \end{proof}

Combining Lemma \ref{l: Levy one vector} with this inequality, we
obtain the tensorized version of Lemma \ref{l: Levy one vector}.
\begin{corollary}                 \label{c: tenzorized}
 Let $U=(u_{i,j})$
  be any $N \times n$  matrix, and let $V$ be
  a $d \times n$ matrix with independent $\d$ random coordinates.
  Then for any $x \in \R^n$
  \begin{equation}   \label{i: vector-sum}
    \LL_1 \left( (U \rowprod V) x , c d \norm{U \rowprod x}_Q
    \right)
   \le  C 2^d \exp \left( - c'\frac{ d \norm{U \rowprod x^T}_Q^2}{N \norm{U \rowprod x^T}^2} \right).
  \end{equation}
\end{corollary}
\begin{proof}
 The coordinates of the vector $(U \rowprod V) x \in \R^{Nd}$ consist
 of $d$ independent blocks $(U \rowprod \e_1) x \etc (U \rowprod \e_d) x$,
 where $\e_1 \etc \e_d$ are the rows of $V$. The corollary follows
 from Lemma \ref{l: trivial-tensorization}, applied to the random
 variables $s_j=\norm{ (U \rowprod \e_j) x-y_j}_1$, where $y_1 \etc y_d
 \in \R^N$ are any fixed vectors.
\end{proof}

 To prove Theorem \ref{t: L_1 norm bound} we have to bound the
 probability that the matrix $\D_1 \rowprod \ldots \rowprod \D_K$
 maps some vector from the unit sphere into a small $\ell_1$ ball. Before
 doing that, we consider an easier problem of estimating the
 probability that this matrix maps a {\em fixed} vector into a small
 $\ell_1$ ball. We phrase this estimate in terms of the Levy
 concentration function.

 \begin{lemma}           \label{l: one vector}
  Let $U \in \WW_{K-1} \cap \VV_{K-1}$ be a $d^{K-1} \times n$ matrix, and
   let $\D_K$ be a $d \times n$ random matrix with
  independent $\d$ random entries.
  For  any $x \in \R^n$
  \begin{align*}
   &\LL_1 \left( (U \rowprod \D_K)x,
      \tilde{c} d^K \norm{x}_2 \right)
   \\
   &\le  \exp \left(- \frac{c'' d \norm{x}_2^2}{\norm{x}_{\infty}^2} \right)
   + \exp \left(- \frac{c'' d^K}{\log^{K-1} \left( \frac{en \norm{x}_{\infty}^2}{\norm{x}_2^2} \right)} \right).
  \end{align*}
 \end{lemma}

 \begin{proof}
   To use Corollary \ref{c: tenzorized}, we have to estimate the
   $Q$-norm and the operator norms of $U \rowprod x^T$. The
   estimate of the $Q$-norm is given by \eqref{c: large Q norm}.

  To estimate the operator norm, assume that $\norm{x}_2=1$, and set
  $
    s= \left \lfloor \norm{x}_{\infty}^{-2} \right \rfloor.
  $
  Let  $L$ be the maximal number $l$ such that $2^l s \le
  n$, and let $I_0 \etc I_L$ be the blocks of coordinates of $x$ of
  type $s$. Then
    $\norm{x|_{J_l}}_{\infty} \le
   2^{-l} \norm{x}_{\infty}$, and by Lemma \ref{l: block
   decomposition}
   \[
     \sum_{j=0}^{L} |J_l| \cdot \norm{x|_{J_l}}_{\infty}^2 \le 5.
   \]
   Let $y \in \R^n$. By Cauchy--Schwartz inequality, we have
   \begin{align*}
   \norm{(U \rowprod x^T)y}_2^2
   &= \norm{\sum_{l=0}^L(U|_{J_l} \rowprod x^T|_{J_l})y|_{J_l}}_2^2 \\
   &\le \left( \sum_{l=0}^L \norm{U|_{J_l} \rowprod x^T|_{J_l}}_2^2 \right)
      \cdot \left(  \sum_{l=0}^L  \norm{y|_{J_l}}_2^2 \right) \\
   &\le \left( \sum_{l=0}^L \norm{U|_{J_l} \rowprod x^T|_{J_l}}_2^2 \right)
      \cdot  \norm{y}_2^2,
   \end{align*}
   which means
   \[
     \norm{U \rowprod x^T}^2
     \le \sum_{l =0}^{L} \norm{U|_{J_l} \rowprod x^T|_{J_l}}^2
     \le \sum_{l =0}^{L} \norm{U|_{J_l}}^2 \cdot \norm{ x|_{J_l}}_{\infty}^2.
   \]

  Since $U \in \WW_{K-1}$, and $|J_l| \ge |J_1|=s$ for all $l \le
  L$, the previous inequality implies
  \begin{align*}
   \norm{U \rowprod x^T}^2
   &\le C \sum_{l =0}^{L}  \left( d^{K-1}+ |J_l| \cdot
   \log^{K-1} \left(\frac{ e n}{|J_l|} \right) \right)
    \cdot \norm{ x|_{J_l}}_{\infty}^2
    \\
    &\le C \left(  d^{K-1} \norm{x}_{\infty}^2
      + \log^{K-1} \left(\frac{e n}{s} \right) \right).
  \end{align*}
  Therefore, by Corollary \ref{c:
  tenzorized} and condition \eqref{c: large Q norm},
  \begin{align*}
  &\LL_1 \left( (U \rowprod \D_K)x,
      c d^K  \right)
   \\
   &\le \exp \left(- \frac{C d^K }{ d^{K-1} \norm{x}_{\infty}^2
      + \log^{K-1} \left(\frac{e n}{s} \right)} \right).
  \end{align*}
  The lemma follows from an elementary inequality
  $\exp(-\frac{a}{b+c}) \le \exp(-\frac{a}{2 b})+ \exp(-\frac{a}{2c})$.
 \end{proof}

\section{Lower bounds via the chaining argument}   \label{s: chaining}

 To get a global bound for the Levy concentration function using the
 bounds for each fixed vector, we prove a chaining-type
 estimate. Chaining argument is one of the main approaches to obtaining
 bounds for the {\em supremum} of a random process \cite{Ta1}.
 Let $\{X_t \mid t \in T\}$ be a random process indexed by a
 set $T$.
  The chaining method is based on representing $X_t$ as a sum of increments
  and proving an upper estimate for each increment separately, and
  combining these estimates using  the union bound.

  A similar approach, based on passing from a random variable to
   increments can be applied to estimating the {\em infimum} of a
   random process as well. In this case we isolate one ``big''
   increment, whose position in the chain depends on $t$.
    The rest of the increments is divided in two groups. In
   one group the increments are small, and we can bound their
   absolute values above, and use the triangle inequality. The
   increments from the other group may be big, but they belong to a small set of random
   variables. In such situation, we can condition on these
   increments, and obtain a lower bound on the conditional
   probability using the Levy concentration function of the ``big'' increment.
   Then we sum up these conditional probabilities over the
   small set. As usual for the chaining method, this step requires a
   balance between the estimate of the Levy concentration function,
   and the size of the set.

 \begin{lemma}               \label{l: chaining}
  Let  $R>0$, $\a \in (0, 1/2)$ and  let $\{l_j\}_{j=0}^L$ be a sequence
   of natural numbers such that
  $l_0=1$ and $l_{j+1} \ge 2 l_j$ for all $j=0 \etc L$. Set
  $n = \sum_{j=1}^L l_j$. Let $A: \R^n \to \R^N$
  be a random matrix with
  independent columns. Assume that for any $j = 1 \etc L$ there exists $p_j>0$
  such that for
  any $x \in S^{n-1}$ with $|\supp(x)| \le l_j, \
  \norm{x}_\infty \le l_{j-1}^{-1/2}$
  \begin{equation}    \label{c: Levy concentration}
    \LL_1(Ax, R)
    \le p_j \le \left( \frac{6e n}{ l_j \a^j} \right)^{-8 l_j}.
  \end{equation}
  Then for any $y \in \R^N$
  \[
    \P \left( \exists x \in S^{n-1} \ \norm{Ax-y}_1 \le \frac{\a^{L-1} R}{4} \right)
    \le  p_1^{1/2} +\P \left(\norm{A}> \frac{R}{8 \a \sqrt{N}} \right).
  \]
 \end{lemma}

\begin{proof}
 Denote $\norm{A}_{2 \to 1}=\norm{A: B_2^n \to B_1^N}$.
 Let $j \in \{1 \etc L\}$ and let $J$ be a $l_j$-element subset of $\{1 \etc L\}$.
 Denote
 \[
   S_J= \{x \in S^{n-1} \mid |\supp(x) \subset J, \  \norm{x}_{\infty}
   \le l_{j-1}^{-1/2} \}.
 \]
    Set $m_j=\sum_{i=1}^{j-1} l_i$. Since the sequence $\{l_j\}_{j=1}^L$
   increases exponentially, $m_j \le  l_j$.
 We will need the following
 \begin{claim}
   Let $y \in \R^N$.
    Let
 \[
  Q_J=\{w \mid , \norm{w}_2 \le 2 \a^{1-j}, \
    \supp(w) \cap J= \emptyset, \ |\supp(w)| \le m_j \}.
 \]
    Then
    \[
      \P (\exists z \in Q_J+S_J \ \norm{Az-y}_1 \le R- \a
      \norm{A}_{2 \to 1} ) \le p_j^{3/4}.
    \]
 \end{claim}
 By the volumetric estimate we can choose an $(\a/2)$-net  $\mathcal{M}_J$
 in $S_J$ such that $|\mathcal{M}_J| \le (6/\a)^{l_j}$.

 Take any $x \in \mathcal{M}_J$ and $ w \in Q_J$.
 Denote $y'=y-Aw$. Then the vectors $Ax$ and $y'$ are independent.
 Conditioning on the  columns of $A$ with indexes from $\supp(w)$,
  and using \eqref{c: Levy concentration}, we get
 \[
  \P ( \norm{A(w+x)-y}_1 <   R \mid  A|_{J^c})
  \le \P ( \norm{Ax-y'}_1 <   R \mid  A|_{J^c}) \le p_j.
 \]
 Taking the expectation with respect to $A|_{J^c}$ yields
 \[
  \P ( \norm{A(w+x)-y}_1 <   R)  \le p_j.
 \]

 The volumetric estimate guarantees the existence of a $(\a/2)$-net
 $\NN_J$ in $Q_J$ such that
 \[
  |\NN_J|
  \le \binom{n}{m_j} \left(6 \a^{-j} \right)^{m_j}
  \le \left( \frac{6e n   }{\a^{j} m_j} \right)^{m_j}.
 \]

 Since $m_j \le l_j$,  the last quantity does not exceed $ \left( \frac{6e n }{\a^{j} l_j} \right)^{l_j}$.
 By the union bound and assumption \eqref{c: Levy concentration},
 \begin{align*}
  &\P (\exists x \in \mathcal{M}_J \ \exists w \in \NN_J \   \norm{A(w+x)-y}_1 <
  R) \\
  &\le |\NN_J| \cdot |\mathcal{M}_J| \cdot p_j
  \le \left( \frac{6e n   }{\a^{j} l_j} \right)^{l_j} \cdot
  \left( \frac{6}{\a} \right)^{l_j} \cdot p_j
  \le p_j^{3/4}.
 \end{align*}
 Assume that a point $x'+w' \in S_J+Q_J$ satisfies
 $ \norm{A(w'+x')-y'}_1 < R- \a \norm{A}_{2 \to 1}$.
 Then, approximating it by a point $x+w \in \mathcal{M}_J + \NN_J$,
 such that $\norm{x'+w'-x-w}_2 <\a$,
 we get  $ \norm{A(w+x)-y}_1 < R$.
 This, in combination with the probability estimate above, proves the claim.

 Applying the union bound again, we see that the event
 \begin{align*}
   \Om
   = \{ &\exists j \le n \ \exists J \subset \{1 \etc n\} \ |J|=l_j \
    \exists x \in S_J \ \exists w \in Q_J \\
   &\norm{A(w+x)-y}_1 < R- \a\norm{A}_{2 \to 1} \}
 \end{align*}
 satisfies
 \[
  \P(\Om) \le \sum_{j=1}^L \sum_{|J|=l_j} p_j^{3/4}
  \le \max_{j=1 \etc L} \binom{n}{l_j} p_j^{1/4} \cdot \sum_{j=1}^L p_j^{1/2}.
 \]
 By condition \eqref{c: Levy concentration}, $\binom{n}{l_j} p_j^{1/4} \le p_j^{1/8} \le 1/2$. The same condition and the exponential growth of $l_j$ show also that the sequence $\{p_j^{1/2}\}_{j=1}^L$ decays exponentially, and $\sum_{j=1}^L p_j^{1/2} \le 2 p_1^{1/2}$. This implies
 \begin{equation}\label{prob Omega}
  \P(\Om) \le p_1^{1/2}.
 \end{equation}

 Now let $x \in S^{n-1}$ be any point. Let $\pi: \{1 \etc n \} \to \{1 \etc n
 \}$ be a permutation rearranging the absolute values of the
 coordinates of $x$ in the non-increasing order: $|x_{\pi(1)}| \ge
 |x_{\pi(2)}| \ge \ldots \ge |x_{\pi(n)}|$. Let $I_1 \cup I_2 \cap
 \ldots \cap I_L = \{1 \etc n\}$ be the decomposition of $\{1 \etc n
 \}$ into a disjoint union of consecutive intervals such that
 $|I_j|=l_j$.
 Set $J_j= \pi^{-1}(I_j)$. In other words, the set $J_1$ contains
 $l_1$ largest coordinates of $x$, $J_2$ contains $l_2$ next largest
 etc. Let $x_j$ be the coordinate projection of $x$ to $J_j$, i.e.
 $x_j(i) = x(i) \cdot \mathbf{1}_{J_j}(i)$. Since the largest
 coordinate of $x_j$ has the position $\sum_{i=1}^{j-1}l_i+1$ in the
 non-increasing rearrangement, and $\norm{x}_2=1$, we conclude that
 \[
   \norm{x_j}_{\infty} \le \left( \sum_{i=1}^{j-1}l_i+1
   \right)^{-1/2}
   \le l_{j-1}^{-1/2}.
 \]
 If for all $j=1 \etc L$ $\norm{x_j}_2 \le \a^{j-1} /2$, then
 \[
   \norm{x}_2
   \le  \sum_{j=1}^L \norm{x_j}_2
   \le \frac{1}{2} \cdot \frac{1}{1-\a} <1.
 \]
 Hence, there exists a $j$ such that $\norm{x_j}_2 > \a^{j-1} /2$.
 Let $j$ be the largest number satisfying this inequality.
 Then the vector $u= \sum_{i=j+1}^{L} x_i$ satisfies
 $
   \norm{u}_2
   \le \sum_{i=j+1}^{L} \norm{x_i}_2
   \le \a^j.
 $

  Assume that $\norm{Ax-y}_1 \le \a^{L-1} (R/2 - 2 \a \norm{A}_{2 \to 1})$.
 Then
 \begin{align*}
   \norm{A (\sum_{i=1}^j x_i) -y}_1
   &\le \a^{j-1} (R/2 - 2 \a \norm{A}_{2 \to 1})
   +\norm{A}_{2 \to 1} \cdot \norm{u}_2  \\
    & \le \a^{j-1} (R/2 -  \a \norm{A}_{2 \to 1}).
 \end{align*}
    Set $J=\supp(x_j)$, $z=x_j/\norm{x_j}_2$, and  $w=( \sum_{i=1}^{j-1} x_i)/\norm{x_j}_2$.
    Since $\norm{x_j}_2 > \a^{j-1} /2$,
    $w \in Q_J$ and the inequality above implies \newline
    $\norm{A(w+z)-y/\norm{x_j}_2}_1 \le R-2 \a \norm{A}_{2 \to 1}$.
    Hence, the assumption above implies that the event $\Om$ occurs.
    Therefore,
  \begin{align*}
    &\P \left( \exists x \in S^{n-1} \ \norm{Ax-y}_1 \le \frac{\a^{L-1}R}{4} \right)
    \\
    &\le \P \left( \exists x \in S^{n-1} \
       \norm{Ax-y}_1 \le \a^{L-1} \Big(\frac{R}{2} -  2 \a \norm{A}_{2 \to 1} \Big)
       \text{ and } \norm{A}_{2 \to 1} \le \frac{R}{8 \a} \right)
    \\
    &\quad + \P \left( \norm{A}_{2 \to 1} > \frac{R}{8 \a} \right)
     \le \P(\Om) + \P \left( \norm{A}_{2 \to 1} > \frac{R}{8 \a}
     \right).
  \end{align*}
  Since $\norm{A}_{2 \to 1} \le \sqrt{N} \norm{A}$, the lemma is proved.
\end{proof}

 \section{Lower bounds for $\ell_1$ and $\ell_2$ norms}
 \label{s: lower bounds}

 In this section we use the chaining lemma \ref{l: chaining} to
 prove Theorems \ref{t: L_1 norm bound} and \ref{t: smallest singular
 value}. Actually we will prove a statement, which is stronger than
 Theorem \ref{t: L_1 norm bound}.

 \begin{theorem}  \label{t: L_1 Levy concentration bound}
 Let $K,q,n,d$ be natural numbers. Assume that
 \[
   n \le \frac{c d^K}{\log_{(q)}d}.
 \]
  and let $\D_1 \etc \D_K$ be  $d \times n$ matrices with
 independent $\d$ random entries. Then for any $y \in \R^{d^K}$
 \[
   \P \left( \exists x \in S^{n-1} \ \norm{(\D_1 \rowprod \ldots \rowprod \D_K)x-y}_1 \le c' d^K \right)
   \le C' \exp \left( - \bar{c} d  \right).
 \]
\end{theorem}

\begin{proof}
 Assume first, that $d^K \ge n \ge d^{K-1/2}$ so the condition of Lemma
 \ref{l: norm} holds for $k=K-1$.
Set $R= \tilde{c} d^{K}$, where $\tilde{c}$ is the constant from
Lemma
 \ref{l: one vector}
 Set $\a=8 \tilde{c}/C'$, where $C'$
 is the constant from Theorem \ref{t: norm-row product}.
 By this Corollary,
 \[
   \P(\norm{\D_1 \rowprod \ldots \rowprod \D_K} > \frac{R}{8\a \sqrt{d^K}})
    \le \exp \left(-cd \right).
 \]
 Denote $U=\D_1 \rowprod \ldots \rowprod \D_{K-1}$, and let $\mathcal{U}$
 be the set of all $d^{K-1} \times n$ matrices $A$ satisfying
 \[
   \P(\norm{A \rowprod \D_K} > \frac{R}{8\a \sqrt{d^K}}) \le
   \exp \left(-c'd \right),
 \]
 where $c'=c/2$. By the Chebychev's inequality
 $\P(U \in \mathcal{U}^c) \le \exp \left(-c'd \right)$.
 Let $y \in \R^{d^K}$.
 By Lemmata \ref{l: norm} and \ref{l: Q norm},
\begin{align*}
  &\P(\exists x \in S^{n-1} \mid  \norm{(\D_1 \rowprod \ldots \rowprod \D_K)x-y}_1 <c d^K)    \\
  &\le \P(\exists x \in S^{n-1} \mid  \norm{(U \rowprod \D_K)x-y}_1 <c d^K
  \text{ and } U \in
  \mathcal{W}_{K-1} \cap \VV \cap \mathcal{U}) \\
  &   \quad  +c e^{-c'' d}.
\end{align*}
 This estimate shows that it is enough to bound the conditional probability
 \[
  \P(\exists x \in S^{n-1} \mid  \norm{(U \rowprod \D_K)x-y}_1 <c d^K \mid U)
 \]
  for all matrices
 $U \in \WW_{K-1} \cap \VV \cap \mathcal{U}$.
 This bound is based on Lemma \ref{l: chaining}.
 Fix a matrix $U \in  \WW_{K-1} \cap \VV \cap \mathcal{U}$ for the rest of the proof.
 Let $L=K+q$. It is enough to define numbers $l_1 \etc l_L \in \N$,
 and $p_1 \etc p_L \in (0,1)$ which
 satisfy the conditions of Lemma \ref{l: chaining}. These numbers
 will be constructed differently for $j \le K$ and $j>K$. The
 difference between these cases stems from the different behavior of
 the bound in Lemma \ref{l: one vector}. For relatively small $l_j$
 the $\ell_{\infty}$ norm of a vector $x$  is large, and the second term
 in Lemma \ref{l: one vector} is negligible, compare to the first
 one. However, for $l_j \ge cd^K/\log^K n$ the picture is opposite,
 and the second term is dominating.

 We consider the case $1 \le j \le K$ first. Set $l_0=1$ and $c_0=1$.
  For $1 \le j \le K$ set
  \begin{equation}  \label{eq: l_j}
   l_j=\left \lfloor \frac{c_j d^j}{\log^j d}  \right \rfloor,
  \end{equation}
  where the constants $c_1 \etc c_K$ will be defined inductively.
  Assume that $c_1 \etc c_{j-1}$ are already defined.
  Applying Lemma \ref{l: one vector} to any vector $x \in S^{n-1}$
  with $\norm{x}_{\infty} \le l_{j-1}^{-1/2}$, we get
 \begin{align*}
   &\P \left( \norm{(U \rowprod \D_K)x-y}_1
     \le c d^K\right)
   \le  \exp \left(- c d l_{j-1} \right)
   + \exp \left(- \frac{c d^K}{\log^{K-1}  n } \right)
   \\
   &\le \exp \left(- \frac{c'_{j-1} d^j}{\log^{j-1} d} \right)=:p_j,
  \end{align*}
  where we can take $c_{j-1}'=c \cdot c_{j-1}/2$.
  Inequality \eqref{c: Levy concentration} reads
  \[
    \frac{c'_{j-1} d^j}{\log^{j-1} d}
    \ge 8 \frac{c_j d^j}{\log^j d} \cdot
      \log \left( \frac{6e n \cdot \log^j d}{c_j  d^j   \a^j} \right).
  \]
  Since $n \le d^K$, this inequality follows from
  \[
   c_{j-1}' \ge  \frac{8 c_j }{\log d} \cdot
      \log \left( \frac{6e d^{K}}{c_j    \a^K} \right).
  \]
  Therefore, we can choose $c_j$ independently of $d$, so that the inequality above is
  satisfied. Thus, the sequence $l_1 \etc l_K$ satisfies condition
  \eqref{c: Levy concentration}. Also, if $d \ge d_0$ for some $d_0$
  depending only on $K$ and $\d$, then $l_{j+1} \ge 2 l_j$ for all
  $j=1 \etc K-1$.

  Let us now define the numbers $l_{K+s}$ for $s=1 \etc q+1$. To
  this end define the
  sequence $\{\b_s\}_{s=0}^{q}$ by induction. Set
  \begin{align*}
   \b_0&=\frac{\log^K d}{c_K} \quad \text{and}\\
   \b_s &= \tilde{c}  \log_{(1)}^K \left(6 e  \b_{s-1}  \right) \
     \text{ for }       1 \le s \le q,
  \end{align*}
  where the number $\tilde{c} \ge 1$ will be chosen below.
  For $0 \le s \le q$ set
  \[
    l_{K+s}=\lfloor d^K/\b_s \rfloor.
  \]
    Note that for $s=0$ this formula agrees with \eqref{eq: l_j}.
    Let $1 \le s \le q$.
  By Lemma \ref{l: one vector}, any vector $x \in S^{n-1}$
  with $\norm{x}_{\infty} \le l_{K+s-1}^{-1/2}$ satisfies
 \begin{align*}
   &\P \left( \norm{(U \rowprod \D_K)x-y}_1
     \le c d^K\right)  \\
   &\le  \exp \left(- c d l_{K+s-1} \right)
   + \exp \left(- \frac{c d^K}{\log^{K-1} \left( \frac{e n }{l_{K+s-1} } \right) } \right)
   \\
   &\le 2 \exp \left(- \frac{c d^K}{\log^{K-1} \left( \frac{e n }{l_{K+s-1} } \right) } \right)=:p_{K+s}.
  \end{align*}
  The last inequality follows from $d l_{K+s-1} > d^K$.
  In this case, condition \eqref{c: Levy concentration} reads
  \[
     \frac{c d^K}{\log^{K-1}  \left( \frac{e n }{l_{K+s-1} } \right) }
    \ge 8 l_{K+s} \cdot \log \left( \frac{6 e n}{l_{K+s}
    \a^{K+s}} \right),
  \]
  which can be rewritten as
  \[
   \frac{c}{\log^{K-1} \left( \frac{e n \b_{s-1}}{d^K} \right)}
   \ge \frac{8}{\b_s} \cdot
     \log \left( \frac{6 e n \b_{s}}{\a^{K+s} d^K}  \right).
  \]
  Since the sequence $\{\b_s \}_{s=0}^q$ is decreasing,
  and $n \le d^K$, the previous inequality holds, provided
  \[
    \b_{s}
    \ge \frac{8}{c}   \log^{K-1} \left( e  \b_{s-1} \right)
    \cdot \left[ \log(6 e \b_{s-1}) + (K+q) \log \frac{1}{\a} \right].
  \]
  Since by the definition of $\b_s$, $\log (e \b_{s-1}) \ge 1$, we
  can choose
  \[
    \tilde{c}=  \frac{8}{c} \cdot (K+q) \log \frac{1}{\a}.
  \]
  The inductive definition of the
  numbers $\b_1 \etc \b_q$ is complete, and the sequences $l_1 \etc
  l_{K+q}, \ p_1 \etc p_{K+q}$ satisfy condition \eqref{c: Levy
  concentration}. Also, if $d \ge d_1$ for some $d_1$ depending only
  $K,q$, and $\d$, then $\b_{s+1} \le \b_s/2$, and so $l_{K+s+1} \ge
  2 l_{K+s}$ for $s=0 \etc q-1$.

 Set $\tilde{n} = \sum_{j=1}^{K+q} l_j$. Then $l_{K+q} \le \tilde{n}
 \le 2 l_{K+q}$. From the definition of $\b_s$ and induction follows
 that $1 \le \b_s \le c' \log_{(s)} d$ for all $s=1 \etc q$.
  Hence, there exists $c>0$ depending only on $K,q,
 \d$ such that
 \[
   \frac{c d^K}{\log_{(q)} d} \le \tilde{n} \le d^K.
 \]
    Thus, for $d \ge \max(d_0,d_1)$, and $n = \tilde{n}$, the
    assertion of Theorem \ref{t: L_1 Levy concentration bound}
    follows from Lemma \ref{l: chaining}. It automatically extends
    to all $n \le \tilde{n}$, since for any $y \in \R^{d^K}$ the
    quantity
    \[
      \min_{x \in S^{\tilde{n}-1}} \norm{(\D_1 \rowprod \ldots \rowprod \D_K)x-y}_1
    \]
    can only increase, if we take the minimum over $S^{\tilde{n}-1}
    \cap \R^n$, instead of the whole sphere, in other words, if we
    consider a submatrix of $\D_1 \rowprod \ldots \rowprod \D_K$ consisting of $n$ first
    columns. It can be also automatically extended to the case $d <
    \max(d_0,d_1)$ by choosing a large constant $C'$ in the
    formulation of the Theorem. The proof is now complete.
\end{proof}
\begin{remark} \label{r: zero column}
 The probability estimate of Theorem \ref{t: L_1 Levy concentration bound} is actually
 optimal. Indeed, let $y=0$, and assume that the entries of the matrices $\D_1 \etc
 \D_K$ are i.i.d. random variables taking values $0,1,-1$ with
 probability $1/3$ each. Then with probability $(1/3)^d$, the first
 column of $\D_1$ is 0, and so the first column of $\D_1 \rowprod
 \ldots \rowprod \D_K$ is 0 as well.
\end{remark}

 We conclude with the proof of Theorem \ref{t: smallest singular value}.
 Set $\tilde{\D}= \D_1 \rowprod \ldots \rowprod \D_K$. By Theorem \ref{t: L_1 norm
 bound}, with probability at least $1-\exp (-cd)$,
  \[
   \norm{\tilde{\D}} \le C' d^{K/2} \quad \text{ and }
  \forall x \in S^{n-1} \ \norm{\tilde{\D}x}_1 \ge c' d^K.
  \]
 Then for any $x \in S^{n-1}$
 \[
  c'd^K \le \norm{\tilde{\D} x}_1 \le d^{K/2} \norm{\tilde{\D} x}_2
   \le d^{K/2} \norm{\tilde{\D} }_2 \cdot \norm{ x}_2 \le C' d^K,
 \]
 so all these norms are equivalent.
 Comparison between the first and the third term of this inequality
 implies Theorem \ref{t: smallest singular value}.
 Moreover, as in \cite{R1}, we can conclude
 that
 $\tilde{\D} \R^n$ is a Kashin subspace of $\R^{d^K}$, i.e. the
 $L_1$ and $L_2$ norms are equivalent on it. More precisely, this
 establishes the following corollary.

 \begin{corollary}  \label{cor: Kashin's subspace}
  Under the conditions of Theorem \ref{t: L_1 norm bound}
  \[
   \P(\forall y \in \tilde{\D} \R^n \ \norm{y}_1 \le d^{K/2} \norm{y}_2 \le C'' \norm{y}_1)
   \ge 1 - \exp \left(-cd \right).
  \]
 \end{corollary}

\end{document}